\documentclass{imanum}

\jno{drnxxx}
\received{}
\revised{}

\usepackage{graphicx,color}
\usepackage[dvips]{epsfig}
\usepackage{amsmath,amssymb,amsfonts,amsthm,mathrsfs}
\usepackage{mathdots}
\usepackage{rotating}
\usepackage{cite}
\usepackage{todonotes}

\def\tablenotes{\bgroup\parfillskip=0pt plus 1fil
\leftskip=0pt\relax \rightskip=0pt
\vskip2pt\footnotesize}
\def\endtablenotes{\vskip1pt\egroup}

\def\XXint#1#2#3{{\setbox0=\hbox{$#1{#2#3}{\int}$}
     \vcenter{\hbox{$#2#3$}}\kern-.5\wd0}}

\def\pr(#1){\left({#1}\right)}
\def\br[#1]{\left[{#1}\right]}
\def\fbr[#1]{\!\left[{#1}\right]\!}

\def\CC{{\cal C}}

\def\addtab#1={#1\;&=}

\def\xc{{\bf x}_N^{\rm cheb}}
\def\tc{\boldsymbol{\theta}_N^{\rm cheb}}
\def\cc{{\bf c}_N^{\rm cheb}}
\def\cj{{\bf c}_N^{\rm jac}}
\def\wj#1{{\bf w}_N^{(#1)}}
\def\cl{{\bf c}_N^{\rm leg}}
\def\TN{{\bf T}_N(\xc)}
\def\PN#1{{\bf P}_N^{(#1)}(\xc)}

\def\xct{{\bf x}_{2N}^{\rm cheb}}
\def\tct{\boldsymbol{\theta}_{2N}^{\rm cheb}}
\def\wjt#1{{\bf w}_{2N}^{(#1)}}
\def\TNt{{\bf T}_{2N}(\xct)}
\def\PNt#1{{\bf P}_{2N}^{(#1)}(\xct)}

\newcommand{\diag}{\operatorname{diag}}
\newcommand{\logop}{\operatorname{log1p}}

\newcommand{\argmin}{\operatornamewithlimits{arg\,min}}

\begin{document}

\title{On the use of Hahn's asymptotic formula and stabilized recurrence for a fast, simple, and stable Chebyshev--Jacobi transform}
\shorttitle{The Chebyshev--Jacobi Transform}

\author{%
{\sc Richard Mika\"el Slevinsky\thanks{Corresponding author. Email: Richard.Slevinsky@maths.ox.ac.uk}}\\[2pt]
Mathematical Institute, University of Oxford, Woodstock Road, Oxford OX2 6GG, UK}
\shortauthorlist{R. M. Slevinsky}

\maketitle

\begin{abstract}
{We describe a fast, simple, and stable transform of Chebyshev expansion coefficients to Jacobi expansion coefficients and its inverse based on the numerical evaluation of Jacobi expansions at the Chebyshev--Lobatto points. This is achieved via a decomposition of Hahn's interior asymptotic formula into a small sum of diagonally scaled discrete sine and cosine transforms and the use of stable recurrence relations. It is known that the Clenshaw--Smith algorithm is not uniformly stable on the entire interval of orthogonality. Therefore, Reinsch's modification is extended for Jacobi polynomials and employed near the endpoints to improve numerical stability.}
{Chebyshev expansions; Jacobi expansions; fast Fourier transform; asymptotic approximations.}
\end{abstract}

\section{Introduction}

Chebyshev expansions:
\begin{equation}\label{eq:Tn}
p_N(x) = \sum_{n=0}^N c_n^{\rm cheb}T_n(x),\qquad x\in[-1,1],
\end{equation}
where $T_n(\cos\theta) = \cos(n\theta)$ are ubiquitous in numerical analysis, approximation theory and pseudo-spectral methods for their near-best approximation, fast evaluation via the discrete cosine transform, and fast linear algebra for Chebyshev spectral methods, among the many other properties that facilitate their convenient use~(see e.g.~\citet{Mason-Handscomb-02,Olver-et-al-NIST-10,Trefethen-12}).

Jacobi expansions:
\begin{equation}\label{eq:Pabn}
p_N(x) = \sum_{n=0}^N c_n^{\rm jac}P_n^{(\alpha,\beta)}(x),\qquad x\in[-1,1],\quad \alpha,\beta>-1,
\end{equation}
also have useful properties. The Jacobi polynomials are orthogonal with respect to $L^2([-1,1],w^{(\alpha,\beta)}(x){\rm\,d}x)$, where $w^{(\alpha,\beta)}(x) = (1-x)^\alpha(1+x)^\beta$ is the Jacobi weight. Jacobi expansions are therefore useful in pseudo-spectral methods where it is more natural to measure the error in Jacobi weighted Hilbert spaces~(see~\citet{Li-Shen-79-1621-10}). As well~\citet{Wimp-McCabe-Connor-82-447-97} show that the Jacobi weighted finite Hilbert and Cauchy transforms are diagonalized by Jacobi polynomials.

For $N\in\mathbb{N}$, define $\tc$ as the vector of $N+1$ equally spaced angles:
\begin{equation}
[\tc]_n = \tfrac{\pi n}{N},\qquad n=0,\ldots,N,
\end{equation}
and the vector of $N+1$ Chebyshev--Lobatto points $\xc = \cos\tc$. We express the vectors of the evaluation of the expansion~\eqref{eq:Tn} and~\eqref{eq:Pabn} at ${\bf x}_N^{\rm cheb}$ as the equality of the matrix-vector products:
\begin{equation}\label{eq:PTNPN}
p_N(\xc) = \TN\cc = \PN{\alpha,\beta} \cj,
\end{equation}
where the entries of the matrices are:
\begin{equation}
\left[\TN\right]_{i,j} = T_{i-1}([\xc]_{j-1}),\qquad \left[\PN{\alpha,\beta}\right]_{i,j} = P_{i-1}^{(\alpha,\beta)}([\xc]_{j-1}).
\end{equation}
We define the forward Chebyshev--Jacobi transform to be:
\begin{equation}\label{eq:CJT}
\cc = \TN^{-1}\PN{\alpha,\beta}\cj,
\end{equation}
and the inverse Chebyshev--Jacobi transform by:
\begin{equation}\label{eq:iCJT}
\cj = \PN{\alpha,\beta}^{-1}\TN\cc.
\end{equation}

\subsection{Previous work on computing Legendre, Gegenbauer, and Jacobi expansion coefficients}

The origins of the method proposed and analyzed in this paper start with the fast eigenfunction transform of~\citet{Orszag-13-86}. The novelty of his approach, which is improved by~\citet{Mori-Suda-Sugihara-40-3612-99} is the observation that, for large $N$, the matrix $\PN{0,0}$ is well approximated by a small sum of diagonally scaled Discrete Cosine Transforms of type-I (DCT-I's) and Discrete Sine Transforms of type-I (DST-I's). However, by not accounting for the region in the $N$-$x$ plane where the matrix significantly differs from the interior asymptotics, their initial advances were unstable.

Families of orthogonal polynomials are related by the so-called connection coefficients~\citep[p. 357]{Andrews-Askey-Roy-98}. The connection coefficients fill in a lower-triangular matrix that allows conversion between two different families of orthogonal polynomials.~\citet{Alpert-Rokhlin-12-158-91} leverage the asymptotically smooth functions which define the connection coefficients between Chebyshev and Legendre polynomials for an ${\cal O}(N\log N)$ hierarchical approach to the Chebyshev--Legendre transform. This hierarchical approach has been extended by~\citet{Keiner-31-2151-09} for expansions in Gegenbauer polynomials.

When transforming polynomial expansions of analytic functions, an alternative approach to hierarchical decomposition of the connection coefficients can be used. With geometric decay in the coefficients of both the source expansion and the target expansion, the algebraic off-diagonal decay of the connection coefficients has been used by~\citet{Cantero-Iserles-50-307-12} and~\citet{Wang-Huybrechs-14a} for ${\cal O}(N\log N+MN)$ Gegenbauer and Jacobi expansion coefficients of analytic functions, where $M\in\mathbb{N}$ is a parameter.

In principle, the hierarchical approach of~\citet{Alpert-Rokhlin-12-158-91} can be adapted to the Jacobi connection coefficients for an ${\cal O}(N\log N)$ algorithm. However, this approach will also be saddled with the same high pre-computation of the hierarchical matrix. Instead, we extend the approach of~\citet{Hale-Townsend-36-A148-14} by developing fast and numerically stable evaluation of Jacobi polynomials at the Chebyshev--Lobatto points. This approach does not have high pre-computation nor does it require analyticity of the function underlying the expansion. Indeed, the transform produces high absolute accuracy for expansion coefficients of a function with any regularity $\CC^\rho[-1,1]$, $\rho\ge0$. In exchange, we accept an asymptotically slower algorithm.

\citet{Hale-Townsend-36-A148-14} advocate for a modification of the approach of~\citet{Mori-Suda-Sugihara-40-3612-99} based on a block partitioning of the matrix $\PN{0,0}$ into an ${\cal O}(\log N/\log\log N)$ number of partitions within which the interior asymptotics of the Legendre polynomials are guaranteed accurate and the remainder of the matrix is evaluated via recurrence relations. The balancing of operations between stable fast transforms in the blocks with the recurrence relations leads to the complexity ${\cal O}(N\log^2N/\log\log N)$. While asymptotically slower than the hierarchical decomposition of~\citet{Alpert-Rokhlin-12-158-91}, Hale and Townsend advocate that the partitioning algorithm is a practical alternative with a smaller setup cost.

\citet{Hale-Townsend-36-A148-14} leave behind a mystery regarding the discrepancy in the numerically computed error in the coefficients and the theoretical estimates based on model coefficients. In particular, they show that for Legendre coefficients $[{\bf c}_N^{\rm leg}]_n = {\cal O}(n^{-r})$, and for some $r\in\mathbb{R}$, the sup-norm in applying $\PN{0,0}$ is asymptotically:
\begin{equation}
\left\|\PN{0,0}\cl\right\|_\infty =
\left\{\begin{array}{cc}
{\cal O}(N^{1-r}), & r<1,\\
{\cal O}(\log N), & r=1,\\
{\cal O}(1), & r > 1,
\end{array}\right.\quad{\rm as}\quad N\to\infty,
\end{equation}
but in their numerical experiments they observed the larger errors:
\begin{equation}
{\rm Observed~Error} = 
\left\{\begin{array}{cc}
{\cal O}(N^{\frac{3}{2}-r}/\log N), & r=0,\frac{1}{2},1,\\
{\cal O}(1), & r = \frac{3}{2},
\end{array}\right.\quad{\rm as}\quad N\to\infty.
\end{equation}

For the Chebyshev--Legendre transform and more generally for the Chebyshev--Jacobi transform, this mystery is solved here by an extension of Reinsch's modification of the Clenshaw--Smith algorithm to the Jacobi polynomials. It is known that the Clenshaw--Smith algorithm is not uniformly stable on the entire interval of orthogonality, i.e. the error bound of the recurrence relation is spatially dependent. In particular, the loss of accuracy near the endpoints of the interval $[-1,1]$ is significant. Reinsch suggested a modification of Clenshaw's algorithm near the endpoints; the modification is extended by~\citet{Levrie-Piessens-74-85} to the Clenshaw--Smith algorithm for Legendre, ultraspherical, and Laguerre polynomials; and here, we extend it to the Jacobi polynomials.

\subsection{General definitions and properties}

The Gamma function is defined for all $\Re z>0$ by~\citet{Abramowitz-Stegun-65}:
\begin{equation}
\Gamma(z) = \int_0^\infty x^{z-1} e^{-x}{\rm\,d}x,
\end{equation}
and it is analytically continued to $z\in\mathbb{C}\setminus\{-\mathbb{N}_0\}$ by the property $\Gamma(z+1) = z\Gamma(z)$.

The Pochhammer symbol is then defined by~\citet{Abramowitz-Stegun-65}:
\begin{equation}
(x)_n = \dfrac{\Gamma(x+n)}{\Gamma(x)},
\end{equation}
and the beta function is defined similarly by~\citet{Abramowitz-Stegun-65}:
\begin{equation}
{\rm B}(x,y) = \dfrac{\Gamma(x)\Gamma(y)}{\Gamma(x+y)}.
\end{equation}

Jacobi polynomials have the Rodrigues formula~\citep[\S 18.5]{Olver-et-al-NIST-10}:
\begin{equation}
P_n^{(\alpha,\beta)}(x) = \dfrac{(-1)^n}{2^n n!}(1-x)^{-\alpha}(1+x)^{-\beta}\dfrac{{\rm d}^n}{{\rm d}x^n}\left((1-x)^{\alpha}(1+x)^{\beta}(1-x^2)^n\right);
\end{equation}
their values at $x=\pm1$ are known:
\begin{equation}\label{eq:Pnendpts}
P_n^{(\alpha,\beta)}(1) = \binom{n+\alpha}{n},\qquad P_n^{(\alpha,\beta)}(-1) = (-1)^n\binom{n+\beta}{n};
\end{equation}
and, they satisfy the symmetry relation:
\begin{equation}\label{eq:Pnsym}
P_n^{(\alpha,\beta)}(x) = (-1)^n P_n^{(\beta,\alpha)}(-x).
\end{equation}

Their three-term recurrence relation is given by:
\begin{equation}
P_{n+1}^{(\alpha,\beta)}(x) = (A_nx+B_n)P_n^{(\alpha,\beta)}(x) - C_n P_{n-1}^{(\alpha,\beta}(x),\qquad P_{-1}^{(\alpha,\beta)}(x) = 0,\quad P_0^{(\alpha,\beta)}(x) = 1,
\end{equation}
where the recurrence coefficients are given by~\citep[\S 18.9.2]{Olver-et-al-NIST-10}:
\begin{align}
A_n & = \dfrac{(2n+\alpha+\beta+1)(2n+\alpha+\beta+2)}{2(n+1)(n+\alpha+\beta+1)},\\
B_n & = \dfrac{(\alpha^2-\beta^2)(2n+\alpha+\beta+1)}{2(n+1)(n+\alpha+\beta+1)(2n+\alpha+\beta)},\\
C_n & = \dfrac{(n+\alpha)(n+\beta)(2n+\alpha+\beta+2)}{(n+1)(n+\alpha+\beta+1)(2n+\alpha+\beta)}.
\end{align}
The relation between Jacobi polynomials of differing parameters:
\begin{equation}
(\alpha+\beta+2n+1)P_n^{(\alpha,\beta)}(x) = (\alpha+\beta+n+1)P_n^{(\alpha,\beta+1)}(x) + (\alpha+n)P_{n-1}^{(\alpha,\beta+1)}(x),
\end{equation}
combined with the symmetry relation~\eqref{eq:Pnsym}, allows for integer-valued increments and decrements of parameters with linear complexity in the degree.

\begin{lemma}[\citet{Wang-Huybrechs-14a}]
Assume that:
\begin{equation}
P_n^{(\gamma,\delta)}(x) = \sum_{k=0}^n c_{n,k}^{(\alpha,\beta,\gamma,\delta)}P_k^{(\alpha,\beta)}(x).
\end{equation}
Then the coefficients $c_{n,k}^{(\alpha,\beta,\gamma,\delta)}$ are given by:
\begin{align}
c_{n,k}^{(\alpha,\beta,\gamma,\delta)} & = \dfrac{(n+\gamma+\delta+1)_k(k+\gamma+1)_{n-k}(2k+\alpha+\beta+1)\Gamma(k+\alpha+\beta+1)}{(n-k)!\Gamma(2k+\alpha+\beta+2)}\nonumber\\
& \quad \times \,_3F_2\left( 
\begin{array}{l}
k-n,n+k+\gamma+\delta+1,k+\alpha +1\\
k+\gamma+1, 2k+\alpha+\beta+2
\end{array};
1\right),\label{eq:Pnabcdconnection}
\end{align}
where $\,_3F_2$ is a generalized hypergeometric function~\citep[\S 16.2.1]{Olver-et-al-NIST-10}.
\end{lemma}

\section{The forward transform: Jacobi to Chebyshev}

In this section, we extend the algorithm of~\citet{Hale-Townsend-36-A148-14} for the Chebyshev--Legendre transform to the Chebyshev--Jacobi transform by deriving a fast algorithm to compute:
\begin{equation*}
\cc = \TN^{-1}\PN{\alpha,\beta}\cj.
\end{equation*}
$\TN$ is a diagonally scaled DCT-I that can be applied and inverted in ${\cal O}(N\log N)$ operations. 

\subsection{Interior asymptotics of Jacobi polynomials}

The interior asymptotics of Jacobi polynomials are given by~\citet{Hahn-171-201-80}. Given $M\in\mathbb{N}_0$:
\begin{equation}\label{eq:Jacasy}
P_n^{(\alpha,\beta)}(\cos\theta) = \sum_{m=0}^{M-1} C_{n,m}^{\alpha,\beta}f_m(\theta) + R_{n,M}^{\alpha,\beta}(\theta),
\end{equation}
where:
\begin{align}
C_{n,m}^{\alpha,\beta} & = \frac{2^{2n-m+\alpha+\beta+1} {\rm\,B}(n+\alpha+1,n+\beta+1)}{\pi(2n+\alpha+\beta+2)_m},\label{eq:Cnmab}\\
f_m(\theta) & = \sum_{l=0}^m\dfrac{(\tfrac{1}{2}+\alpha)_l(\tfrac{1}{2}-\alpha)_l(\tfrac{1}{2}+\beta)_{m-l}(\tfrac{1}{2}-\beta)_{m-l}}{l!(m-l)!}\dfrac{\cos\theta_{n,m,l}}{\sin^{l+\alpha+\frac{1}{2}}\left(\frac{\theta}{2}\right)\cos^{m-l+\beta+\frac{1}{2}}\left(\frac{\theta}{2}\right)},\\
\theta_{n,m,l} & = \tfrac{1}{2}(2n+\alpha+\beta+m+1)\theta - (\alpha+l+\tfrac{1}{2})\tfrac{\pi}{2},
\end{align}
and where $x=\cos\theta$. For $(\alpha,\beta)\in(-\tfrac{1}{2},\tfrac{1}{2}]^2$, and for $n\ge 2$, the remainder $R_{n,M}^{\alpha,\beta}(\theta)$ is bounded by twice the magnitude of the first neglected term in the summation, and for $\theta\in[\tfrac{\pi}{3},\tfrac{2\pi}{3}]$ the summation converges as $M\to\infty$.

Rewriting $\theta_{n,m,l}$ as:
\begin{align}
\theta_{n,m,l} & = \tfrac{1}{2}(2n+\alpha+\beta+m+1)\theta - (\alpha+l+\tfrac{1}{2})\tfrac{\pi}{2},\\
& = n\theta + (\alpha+\beta+m+1)\tfrac{\theta}{2} - (\alpha+l+\tfrac{1}{2})\tfrac{\pi}{2},\\
& = n\theta -\theta_{m,l},
\end{align}
allows us to insert the cosine addition formula into the asymptotic formula~\eqref{eq:Jacasy}. The result is:
\begin{equation}\label{eq:Jacasyuv}
P_n^{(\alpha,\beta)}(\cos\theta) = \sum_{m=0}^{M-1} \left(u_m(\theta) \cos n\theta + v_m(\theta)\sin n\theta\right)C_{n,m}^{\alpha,\beta} + R_{n,M}^{\alpha,\beta}(\theta),
\end{equation}
where:
\begin{align}
u_m(\theta) & = \sum_{l=0}^m\dfrac{(\tfrac{1}{2}+\alpha)_l(\tfrac{1}{2}-\alpha)_l(\tfrac{1}{2}+\beta)_{m-l}(\tfrac{1}{2}-\beta)_{m-l}}{l!(m-l)!}\dfrac{\cos\theta_{m,l}}{\sin^{l+\alpha+\frac{1}{2}}\left(\frac{\theta}{2}\right)\cos^{m-l+\beta+\frac{1}{2}}\left(\frac{\theta}{2}\right)},\\
v_m(\theta) & = \sum_{l=0}^m\dfrac{(\tfrac{1}{2}+\alpha)_l(\tfrac{1}{2}-\alpha)_l(\tfrac{1}{2}+\beta)_{m-l}(\tfrac{1}{2}-\beta)_{m-l}}{l!(m-l)!}\dfrac{\sin\theta_{m,l}}{\sin^{l+\alpha+\frac{1}{2}}\left(\frac{\theta}{2}\right)\cos^{m-l+\beta+\frac{1}{2}}\left(\frac{\theta}{2}\right)}.
\end{align}
Since in~\eqref{eq:Jacasyuv}, $\cos n\theta$ and $\sin n\theta$ are the only terms that depend simultaneously and inextricably on both $n$ and $\theta$, the matrix $\PN{\alpha,\beta}$ of~\eqref{eq:PTNPN} can be expressed in the compact form:
\begin{equation}\label{eq:PNASY}
\PN{\alpha,\beta}^{\rm ASY} = \sum_{m=0}^{M-1} \left( {\bf D}_{u_m(\tc)} \TN + {\bf D}_{v_m(\tc)} \sin(\tc [0,\ldots,N]^\top) \right){\bf D}_{C_{n,m}^{\alpha,\beta}} + {\bf R}_M^{\alpha,\beta}(\tc).
\end{equation}
Here, ${\bf D}_{u_m(\tc)}$ and ${\bf D}_{v_m(\tc)}$ denote diagonal matrices whose entries correspond to $u_m$ and $v_m$ evaluated at the equally spaced angles $\tc$, ${\bf D}_{C_{n,m}^{\alpha,\beta}}$ is the diagonal matrix whose entries consist of $C_{n,m}^{\alpha,\beta}$ for $n=0,\ldots,N$, and ${\bf R}_M^{\alpha,\beta}(\tc)$ is the matrix of remainders. Since $\TN$ is a diagonally scaled DCT-I and $\sin(\tc [0,\ldots,N]^\top)$ is a diagonally scaled DST-I bordered by zeros, the matrix $\PN{\alpha,\beta}^{\rm ASY}$ can be applied in ${\cal O}(M N\log N + M^2N)$ operations. However, for low degree or for $\theta\approx0$ or $\theta\approx\pi$, the approximation of $\PN{\alpha,\beta}$ by $\PN{\alpha,\beta}^{\rm ASY}$ incurs unacceptably large error. Therefore, we restrict the applicability of the matrix $\PN{\alpha,\beta}^{\rm ASY}$ to the region in the $n$-$\theta$ plane where the remainder is guaranteed to be below a tolerance $\varepsilon$ and we use recurrence relations to stably fill in\footnote{N.B. the Clenshaw--Smith algorithm~\citet{Clenshaw-9-118-55,Smith-19-33-65} for evaluation of polynomials in orthogonal polynomial bases is used rather than explicitly filling in the matrix.} the remaining entries of the matrix $\PN{\alpha,\beta}$.

\subsection{Partitioning the matrix $\PN{\alpha,\beta}$}

For $(\alpha,\beta)\in(-\tfrac{1}{2},\tfrac{1}{2}]^2$ and for $n\ge 2$, the remainder in~\eqref{eq:Jacasy} is bounded by:
\begin{equation}
|R_{n,M}^{\alpha,\beta}(\theta)| < 2C_{n,M}^{\alpha,\beta}|f_M(\theta)|.
\end{equation}
For large $n$, the following leading order asymptotics are valid:
\begin{align}
2C_{n,M}^{\alpha,\beta} & = \dfrac{2^{2n-M+\alpha+\beta+2}}{\pi}\dfrac{\Gamma(n+\alpha+1)\Gamma(n+\beta+1)}{\Gamma(2n+\alpha+\beta+M+2)},\\
& \sim \dfrac{2^{2n-M+\alpha+\beta+2}}{\pi}\dfrac{\sqrt{2\pi}\sqrt{n+\alpha}(\frac{n+\alpha}{e})^{n+\alpha}  \sqrt{2\pi}\sqrt{n+\beta}(\frac{n+\beta}{e})^{n+\beta}}{\sqrt{2\pi}\sqrt{2n+\alpha+\beta+M+1}(\frac{2n+\alpha+\beta+M+1}{e})^{2n+\alpha+\beta+M+1}},\\
& \sim \dfrac{1}{2^{2M-1}\sqrt{\pi}n^{M+\frac{1}{2}}},\quad{\rm as}\quad n\to\infty.
\end{align}
Therefore, if we set the remainder to $\varepsilon$, this will define a curve in the $n$-$\theta$ plane for every $M$ given by:
\begin{equation}\label{eq:asynthetapartition}
n \approx \left\lfloor\left( \dfrac{\varepsilon 2^{2M-1}\sqrt{\pi}}{|f_M(\theta)|}\right)^{-\frac{1}{M+\frac{1}{2}}}\right\rfloor.
\end{equation}

For $\theta\in(0,\pi)$, $|f_m(\theta)|$ is bounded by its envelope:
\begin{equation}
g_m(\theta) = \sum_{l=0}^m\dfrac{(\tfrac{1}{2}+\alpha)_l(\tfrac{1}{2}-\alpha)_l(\tfrac{1}{2}+\beta)_{m-l}(\tfrac{1}{2}-\beta)_{m-l}}{l!(m-l)!}\dfrac{\cos^{l-m-\beta-\frac{1}{2}}\left(\frac{\theta}{2}\right)}{\sin^{l+\alpha+\frac{1}{2}}\left(\frac{\theta}{2}\right)},
\end{equation}
and since $g_m(\theta)\in\CC(0,\pi)$ and:
\begin{equation}
\lim_{\theta\to0^+}g_m(\theta) = \lim_{\theta\to\pi^{-}}g_m(\theta) = +\infty,
\end{equation}
the Weierstrass extreme value theorem ensures the existence of a global minimizer:
\begin{equation}
\hat{\theta} = \argmin_{\theta\in(0,\pi)}g_m(\theta).
\end{equation}
Therefore, we find the discrete global minimizer:
\begin{equation}
\bar{\theta} = \argmin_{\theta\in\tc}g_m(\theta) \approx \hat{\theta},
\end{equation}
and collect contiguous angles such that the error in evaluating the asymptotic expansion is guaranteed to be below $\varepsilon$.

Following~\citet{Hale-Townsend-36-A148-14}, we define $n_M\in\mathbb{N}_0$:
\begin{equation}
n_M := \left\lfloor\left( \dfrac{\varepsilon 2^{2M-1}\sqrt{\pi}}{|f_M(\pi/2)|}\right)^{-\frac{1}{M+\frac{1}{2}}}\right\rfloor,
\end{equation}
and we set:
\begin{equation}
\alpha_N := \min\left(\frac{1}{\log(N/n_M)},1/2\right),\quad{\rm and}\quad K := \left\lceil \frac{\log(N/n_M)}{\log(1/\alpha_N)} \right\rceil.
\end{equation}
For $k=0,\ldots,K$, while $j_k := \alpha_N^kN > n_M$, we compute the indices $i_k^1,i_k^2$ within which the remainder falls below the tolerance $\varepsilon$ and whose angles bracket the discrete global minimizer $\bar{\theta}$:
\begin{equation}
\left[R_{j_k,M}^{\alpha,\beta}(\tc)\right]_i < \varepsilon \qquad \forall i \in \{i_k^1,\ldots,i_k^2\}, \quad [\tc]_{i_k^1} < \bar{\theta} < [\tc]_{i_k^2}.
\end{equation}

We require the following lemma.

\begin{lemma}\label{lemma:remainder}
Let $(\alpha,\beta)\in(-\tfrac{1}{2},\tfrac{1}{2}]^2$. Then for every $M\ge2$:
\begin{equation}
C_{n+1,M}^{\alpha,\beta} \le C_{n,M}^{\alpha,\beta}.
\end{equation}
\end{lemma}
\begin{proof}
We have:
\begin{equation}
\inf_{(\alpha,\beta)\in(-\frac{1}{2},\frac{1}{2}]^2}(\alpha+\beta) = -1,\qquad \max_{\alpha\in(-\frac{1}{2},\frac{1}{2}]}\alpha = \frac{1}{2},\qquad \max_{\beta\in(-\frac{1}{2},\frac{1}{2}]}\beta = \frac{1}{2}.
\end{equation}
Using~\eqref{eq:Cnmab}:
\begin{equation}
C_{n+1,M}^{\alpha,\beta} = \dfrac{(2n+2\alpha+2)(2n+2\beta+2)C_{n,M}^{\alpha,\beta}}{(2n+\alpha+\beta+M+3)(2n+\alpha+\beta+M+2)} \le \dfrac{(2n+3)^2C_{n,M}^{\alpha,\beta}}{(2n+M+2)(2n+M+1)} \le C_{n,M}^{\alpha,\beta}.
\end{equation}
\end{proof}

Lemma~\ref{lemma:remainder} guarantees that if $M\ge2$, the remainder $R_{n,M}^{\alpha,\beta}(\theta)$ is a non-increasing function of $n$. Therefore, the determination of the indices ensures the accuracy of the asymptotic formula within the rectangles $[[\tc]_{i_k^1},[\tc]_{i_k^2}]\times[j_{k},j_{k-1}]$, for $k=1,\ldots,K$, as depicted in Figure~\ref{fig:asypluspartitioning}.

Lastly, for $k=1,\ldots,K$, define:
\begin{equation}
\PN{\alpha,\beta}^{{\rm ASY},k} = \diag({\bf 0}_{0:i_k^1-1},{\bf 1}_{i_k^1:i_k^2},{\bf 0}_{i_k^2+1:N})\PN{\alpha,\beta}^{\rm ASY}\diag({\bf 0}_{0:j_k-1},{\bf 1}_{j_k:j_{k-1}},{\bf 0}_{j_{k-1}+1:N}),
\end{equation}
to be matrices of the asymptotic formula~\eqref{eq:PNASY} within $[[\tc]_{i_k^1},[\tc]_{i_k^2}]\times[j_k,j_{k-1}]$, and:
\begin{equation}
[\PN{\alpha,\beta}^{\rm REC}]_{i,j} = \left\{\begin{array}{cc} \PN{\alpha,\beta}_{i,j}, & i < i_k^1~{\rm or}~i>i_k^2,~j< j_{k-1},~k=1,\ldots,K,\\
\PN{\alpha,\beta}_{i,j}, & i_K^1\le i\le i_K^2,~j<j_K,\\
0, & {\rm otherwise},\end{array}\right.
\end{equation}
which is computed via recurrence relations. Then, the numerically stable formula:
\begin{equation}
\PN{\alpha,\beta} = \PN{\alpha,\beta}^{\rm REC} + \sum_{k=1}^K \PN{\alpha,\beta}^{{\rm ASY},k},
\end{equation}
can be computed in ${\cal O}(N\log^2N/\log\log N)$ operations. For detailed leading order estimates, see Appendix~\ref{appendix:complexity}.

\begin{figure}[htpb]
\begin{center}
\begin{tabular}{cc}
\hspace*{-0.5cm}\includegraphics[width=0.475\textwidth]{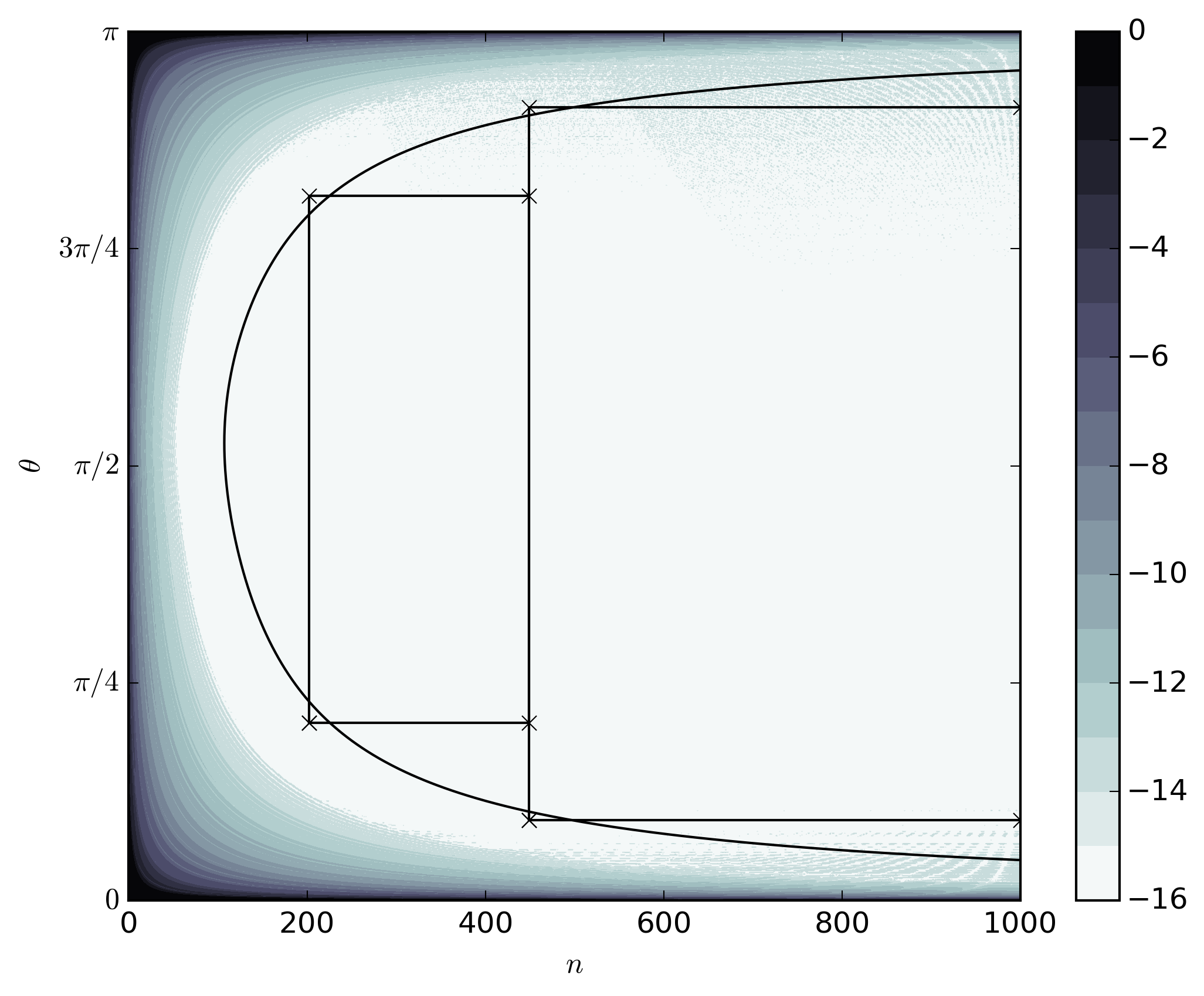}&
\hspace*{-0.5cm}\includegraphics[width=0.475\textwidth]{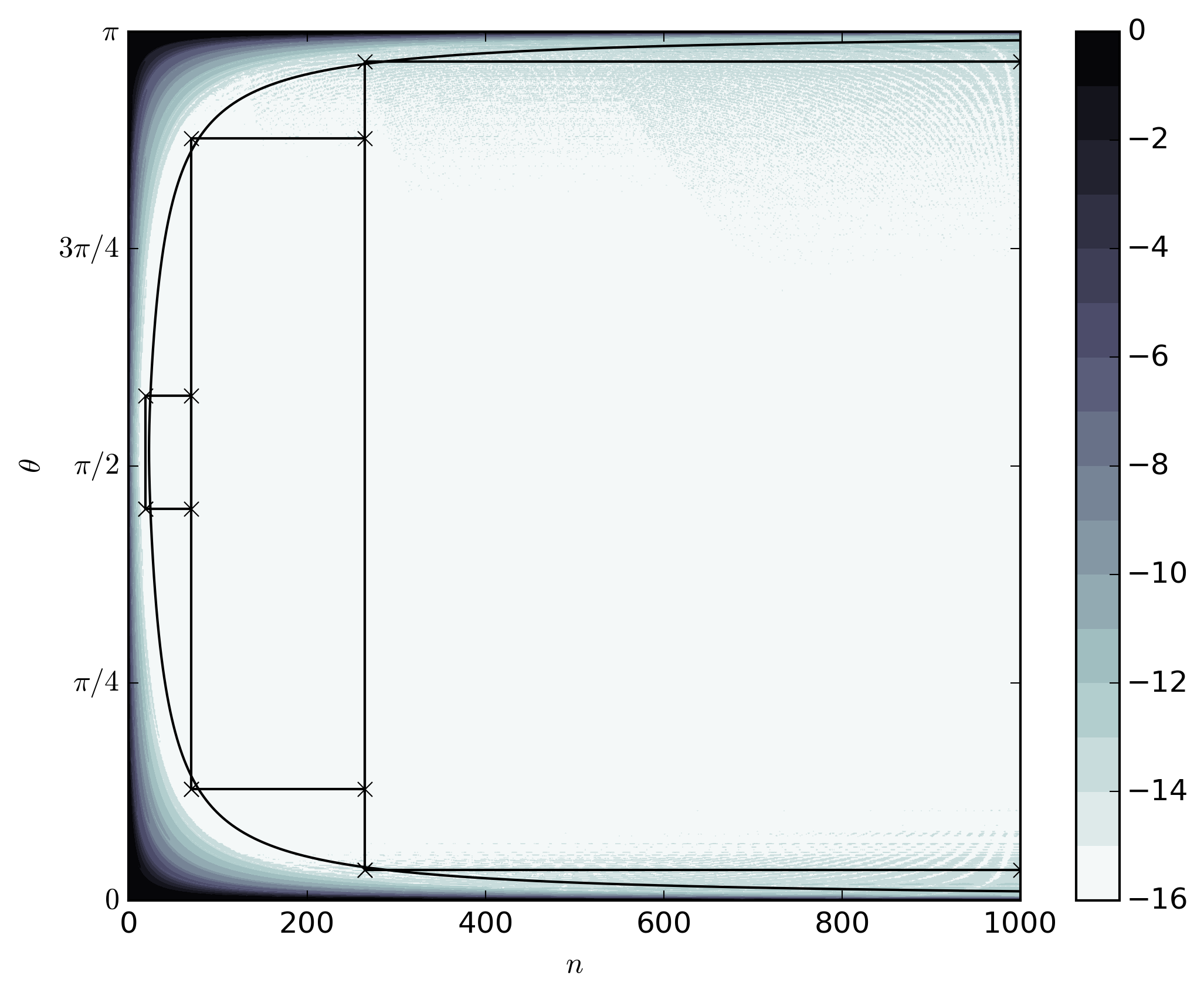}\\
\end{tabular}
\caption{The absolute error in $\PN{1/8,3/8}$ using the asymptotic formula~\eqref{eq:Jacasy} with $M = 7$, left, and $M=13$, right. In both plots: the colour denotes the absolute error on a logarithmic scale; the curves represent the approximate region of accuracy of the asymptotic formula to $\varepsilon\approx 2.2204\times10^{-16}$ determined by~\eqref{eq:asynthetapartition}; whereas, the boxes denote the numerically determined indices $i_k^1,i_k^2,j_k$ for $k=0,1,2$, left, and $k=0,1,2,3$, right, such that the remainder is certainly below $\varepsilon$.}
\label{fig:asypluspartitioning}
\end{center}
\end{figure}

\subsection{Error analysis for model coefficients}

Consider a set of coefficients satisfying $[\cj]_n = {\cal O}(n^{-r})$, for some $r\in\mathbb{R}$. We can estimate the sup-norm of the error in the forward transform~\eqref{eq:CJT} by estimating the error in applying the matrix $\PN{\alpha,\beta}$. Define $D_N^r := \diag(1,1^r,\ldots,N^r)$, then:
\begin{equation}\label{eq:PNerr}
\left\|\PN{\alpha,\beta}\cj\right\|_\infty \le \left\|\PN{\alpha,\beta}D_N^{-r}\right\|_\infty \|D_n^r\cj\|_\infty.
\end{equation}
Since:
\begin{equation}\label{eq:PNbnd}
\max_{x\in\xc} |P_n^{(\alpha,\beta)}(x)| = \max\{|P_n^{(\alpha,\beta)}(1)|,|P_n^{(\alpha,\beta)}(-1)|\} = \max\left\{\binom{n+\alpha}{n},\binom{n+\beta}{n}\right\} = \binom{n+\max\{\alpha,\beta\}}{n},
\end{equation}
we can estimate the first term in~\eqref{eq:PNerr} as follows:
\begin{align}
\left\|\PN{\alpha,\beta}D_N^{-r}\right\|_\infty & = {\cal O}\left(1+\sum_{n=1}^N\binom{n+\max\{\alpha,\beta\}}{n}n^{-r}\right),\quad{\rm as}\quad N\to\infty,\\
& = {\cal O}\left(H_{N,r-\max\{\alpha,\beta\}}\right),\quad{\rm as}\quad N\to\infty,
\end{align}
where $H_{N,r}$ are the generalized harmonic numbers~\citet{Graham-Knuth-Patashnik-89}. Using their asymptotics:
\begin{equation}
\left\|\PN{\alpha,\beta}\cj\right\|_\infty = \left\{\begin{array}{cc}
{\cal O}(N^{1+\max\{\alpha,\beta\}-r}), & r<1+\max\{\alpha,\beta\},\\
{\cal O}(\log N), & r = 1+\max\{\alpha,\beta\},\\
{\cal O}(1), & r > 1+\max\{\alpha,\beta\}.
\end{array}\right.
\end{equation}

\section{The inverse transform: Chebyshev to Jacobi}

It is impractical to compute the inverse transform~\eqref{eq:iCJT}:
\begin{equation*}
\cj = \PN{\alpha,\beta}^{-1} \TN \cc,
\end{equation*}
directly due to the occurrence of the inverse of the matrix $\PN{\alpha,\beta}$. Instead, following~\citet{Hale-Townsend-36-A148-14}, we use the transpose of the asymptotic formula~\eqref{eq:PNASY} in conjunction with the integral definition of the Jacobi coefficients:
\begin{equation}\label{eq:cjintdefn}
[\cj]_n = \dfrac{1}{\mathscr{A}_n^{\alpha,\beta}} \int_{-1}^1 p_N(x) P_n^{(\alpha,\beta)}(x)w^{(\alpha,\beta)}(x){\rm\,d}x,\qquad n=0,\ldots,N,
\end{equation}
where $p_N(x)$ is defined by~\eqref{eq:Tn}, and where $\mathscr{A}_n^{\alpha,\beta}$, defined by~\citep[\S 18.3.1]{Olver-et-al-NIST-10}, are the orthonormalization constants of the Jacobi polynomials:
\begin{equation}\label{eq:Anab}
\mathscr{A}_n^{\alpha,\beta} = \int_{-1}^1 P_n^{(\alpha,\beta)}(x)^2w^{(\alpha,\beta)}(x){\rm\,d}x = \dfrac{2^{\alpha+\beta+1}\Gamma(n+\alpha+1)\Gamma(n+\beta+1)}{(2n+\alpha+\beta+1)\Gamma(n+\alpha+\beta+1)n!}.
\end{equation}
Since $\deg(p_N(x))\le N$, its product with $P_N^{(\alpha,\beta)}(x)$ will be integrated exactly by the $2N+1$-point Clenshaw--Curtis quadrature rule with the Jacobi weight $w^{(\alpha,\beta)}(x)$.

\subsection{Clenshaw--Curtis quadrature}

Clenshaw--Curtis quadrature is a quadrature rule~(see~\citet{Waldvogel-43-001-03,Sommariva-65-682-13}) whose nodes are the $N+1$ Chebyshev--Lobatto points $\xc$. Given a continuous weight function $w(x)\in\CC(-1,1)$ and $\mathbb{P}_N$, the space of algebraic polynomials of degree at most $N$, the weight vector ${\bf w}_N$ is designed by the equality:
\begin{equation}
\int_{-1}^1 f(x)w(x){\rm\,d}x = {\bf w}_N^\top f(\xc),\quad\forall f\in\mathbb{P}_N.
\end{equation}
With the modified Chebyshev moments of the weight function $w(x)$:
\begin{equation}
\mu_n = \int_{-1}^1 T_n(x) w(x){\rm\,d}x,\qquad n=0,\ldots,N,
\end{equation}
the weights ${\bf w}_N$ can be determined via the formula:
\begin{equation}\label{eq:CCweights}
[{\bf w}_N]_n =
\dfrac{1-\tfrac{1}{2}(\delta_{0,n}+\delta_{N,n})}{N}\left\{\displaystyle \mu_0 + (-1)^n\mu_{N} + 2\sum_{k=1}^{N-1} \mu_k\cos[\pi kn/N]\right\}.
\end{equation}
Due to this representation, the ${\cal O}(N\log N)$ computation of the weights from modified Chebyshev moments is achieved via a diagonally scaled DCT-I.

For the Jacobi weight, the modified Chebyshev moments are known explicitly~(see~\citet{Piessens-87}):
\begin{equation}
\mu_n^{(\alpha,\beta)} = \int_{-1}^1 T_n(x) w^{(\alpha,\beta)}(x){\rm\,d}x = 2^{\alpha+\beta+1}{\rm B}(\alpha+1,\beta+1)\,_3F_2\left( 
\begin{array}{l}
n,-n,\alpha +1\\
\frac{1}{2}, \alpha+\beta+2
\end{array};
1\right),
\end{equation}
where $\,_3F_2$ is a generalized hypergeometric function~\citep[\S 16.2.1]{Olver-et-al-NIST-10}. Using Sister Celine's technique~\citep[\S 127]{Rainville-60} or induction~(see~\citet{Xiang-He-Wang-2014-10-14}), a recurrence relation can be derived for the modified moments:
\begin{align}
\mu_0^{(\alpha,\beta)} = 2^{\alpha+\beta+1}{\rm B}(\alpha+1,\beta+1),\quad \mu_1^{(\alpha,\beta)} = \dfrac{\alpha-\beta}{\alpha+\beta+2}\mu_0^{(\alpha,\beta)},&\\
(\alpha+\beta+n+2)\mu_{n+1}^{(\alpha,\beta)} + 2(\alpha-\beta) \mu_n^{(\alpha,\beta)} + (\alpha+\beta-n+2)\mu_{n-1}^{(\alpha,\beta)} = 0,&\quad{\rm for}\quad n>0.
\end{align}
It is known that for $\alpha>\beta$ and $\alpha = -\tfrac{1}{2} + \mathbb{N}_0$ or for $\alpha<\beta$ and $\beta = -\tfrac{1}{2} + \mathbb{N}_0$, neither forward nor backward recurrence is stable. This has been addressed by~\citet{Xiang-He-Wang-2014-10-14} by transforming the initial value problem into a boundary value problem with a sufficiently accurate asymptotic expansion for $\mu_N^{(\alpha,\beta)}$ and subsequent use of Oliver's algorithm~(see~\citet{Oliver-11-349-68}), i.e.~the LU decomposition of a tridiagonal matrix. However, the recurrence relation is stable in the forward direction in the half-open square $(\alpha,\beta)\in(-\tfrac{1}{2},\tfrac{1}{2}]^2$, and in light of the linear complexity of integer-valued decrements, Oliver's algorithm is not required in the present context. Once the modified Chebyshev moments $\mu_n^{(\alpha,\beta)}$ are computed, the Clenshaw--Curtis weights $\wj{\alpha,\beta}$ follow via a diagonally scaled DCT-I.

\subsection{The transpose of the asymptotic formula}

Since $\deg(p_N(x)P_n^{(\alpha,\beta)}(x))\le 2N$, the $2N+1$-point Clenshaw--Curtis quadrature rule yields:
\begin{equation}
[\cj]_n = \dfrac{1}{\mathscr{A}_n^{\alpha,\beta}} (\wjt{\alpha,\beta})^\top (p_N(\xct) P_n^{(\alpha,\beta)}(\xct)),\quad n=0,\ldots,N,
\end{equation}
where $p_N(x)$ is defined by~\eqref{eq:Tn}, and where $\mathscr{A}_n^{\alpha,\beta}$ is given by~\eqref{eq:Anab}. If we let the vector $[{\bf s}_{2N}]_n = (\mathscr{A}_n^{\alpha,\beta})^{-1}$ for $n=0,\ldots,2N$, then we can rewrite this in matrix form:
\begin{equation}\label{eq:CJTT}
\cj = \left[I_{N+1}\,|\,{\bf 0}_N\right] {\bf D}_{{\bf s}_{2N}} \PNt{\alpha,\beta}^\top {\bf D}_{\wjt{\alpha,\beta}} \TNt^\top \begin{bmatrix}I_{N+1}\\ {\bf 0}_N\end{bmatrix} \cc,
\end{equation}
where ${\bf D}_{{\bf s}_{2N}}$ and ${\bf D}_{\wjt{\alpha,\beta}}$ denote diagonal matrices whose entries correspond to ${\bf s}_{2N}$ and $\wjt{\alpha,\beta}$, respectively. Clenshaw--Curtis quadrature allows us to express the Jacobi coefficients in terms of transposed matrices rather than inverse matrices.

In order to complete our formulation, we use the transpose of~\eqref{eq:PNASY}, given by:
\begin{equation}
\PNt{\alpha,\beta}^{\rm ASY,\top} = \sum_{m=0}^{M-1} {\bf D}_{C_{n,m}^{\alpha,\beta}}\left( \TNt^\top {\bf D}_{u_m(\tct)} + \sin(\tct [0,\ldots,2N]^\top)^\top {\bf D}_{v_m(\tct)} \right) + {\bf R}_M^{\alpha,\beta}(\tct)^\top.
\end{equation}
Given the ordering of the points $\xct$, the transposed DCT-I $\TNt^\top$ and the transposed DST-I bordered by zeros $\sin(\tct[0,\ldots,2N]^\top)$ are symmetric, allowing for the same implementation as the forward Chebyshev--Jacobi transform. Similarly, the constants $n_M$, $\alpha_N$, $K$, and the indices $i_k^1$, $i_k^2$, and $j_k$ can be computed as in the forward transform, with the substitution $N\to 2N$. Therefore, the transpose of the asymptotic formula, combined with recurrence relations, can be used for a numerically stable partition and evaluation of the inverse transform.

\subsection{Error analysis for model coefficients}

Consider a set of coefficients satisfying $[\cc]_n = {\cal O}(n^{-r})$, for some $r\in\mathbb{R}$. We can estimate the sup-norm of the error in the inverse transform~\eqref{eq:iCJT} by estimating the error in the transpose formula~\eqref{eq:CJTT}. Using $D_N^r := \diag(1,1^r,\ldots,N^r)$ again, then:
\begin{align}\label{eq:PNTerr}
\left\|\cj\right\|_\infty & = \left\| \left[I_{N+1}\,|\,{\bf 0}_N\right] {\bf D}_{{\bf s}_{2N}} \PNt{\alpha,\beta}^\top {\bf D}_{\wjt{\alpha,\beta}}\TNt^\top \begin{bmatrix}I_{N+1}\\ {\bf 0}_N\end{bmatrix} \cc\right\|_\infty\\
& \le \left\| {\bf D}_{{\bf s}_{2N}} \PNt{\alpha,\beta}^\top\right\|_\infty \left\| {\bf D}_{\wjt{\alpha,\beta}}\right\|_\infty \left\| \TNt^\top\begin{bmatrix}D_N^{-r}\\ {\bf 0}_N\end{bmatrix} \right\|_\infty \|D_N^r\cc\|_\infty.
\end{align}
Using the bound on the Jacobi polynomials~\eqref{eq:PNbnd}, we can formulate asymptotics of the sup-norm involving the transposed matrix $\PNt{\alpha,\beta}^\top$ and its diagonal scaling. Since the inverse squares of the orthonormality constants are asymptotically ${\cal O}(N)$, as can be seen from~\eqref{eq:Anab}, we have:
\begin{align}
\left\| {\bf D}_{{\bf s}_{2N}} \PNt{\alpha,\beta}^\top\right\|_\infty & \le C\left\| D_{2N}^1 \PNt{\alpha,\beta}^\top\right\|_\infty,\quad\textrm{for some }C>0,\\
& \le C\max_{n\in\{0,\ldots,2N\}} \sum_{x\in \xct}(\delta_{n,0}+n)|P_n^{(\alpha,\beta)}(x)|,\\
& \le C\max_{n\in\{0,\ldots,2N\}} 2N (\delta_{n,0}+n)\binom{n+\max\{\alpha,\beta\}}{n},\\
& = {\cal O}\left( N^{2+\max\{\alpha,\beta\}}\right),\quad{\rm as}\quad N\to\infty.
\end{align}
For the Clenshaw--Curtis quadrature weights:
\begin{equation}
\left\| {\bf D}_{\wjt{\alpha,\beta}}\right\|_\infty = {\cal O}(N^{-1}),\quad{\rm as}\quad N\to\infty,
\end{equation}
as can be seen by~\eqref{eq:CCweights}\footnote{Intuitively, since the Clenshaw--Curtis weights must sum to a constant, and not one weight is of paramount importance, this can only occur if they all decay uniformly with ${\cal O}(N^{-1})$.}. Due to the symmetry of $\TNt$, we can conclude:
\begin{equation}
\left\| \TNt^\top\begin{bmatrix}D_N^{-r}\\ {\bf 0}_N\end{bmatrix} \right\|_\infty = 1+H_{N,r},
\end{equation}
or:
\begin{equation}
\left\|\cj\right\|_\infty = \left\{\begin{array}{cc}
{\cal O}(N^{2+\max\{\alpha,\beta\}-r}), & r<1,\\
{\cal O}(N^{1+\max\{\alpha,\beta\}}\log N), & r = 1,\\
{\cal O}(N^{1+\max\{\alpha,\beta\}}), & r > 1.
\end{array}\right.
\end{equation}

This growth rate appears larger than our numerical experiments suggest, and this can be attributed to the overestimation of $\left\| {\bf D}_{{\bf s}_{2N}} \PNt{\alpha,\beta}^\top\right\|_\infty$: the Jacobi polynomials are significantly smaller than the maximum of their endpoints for the majority of the interior of $[-1,1]$. However, without a useful envelope function, we report what can only be an overestimate.

\section{Design and implementation}

As~\citet{Hale-Townsend-36-A148-14} remark, the partitioning implies the algorithm is trivially parallelized. However, of more immediate concern is the application of the same transform to multiple sets of expansion coefficients. Analogous to Fastest Fourier Transform in the West (FFTW) of~\citet{Frigo-Johnson-93-216-05}, we divide the computation into part 1.~Planification and part 2.~Execution:
\begin{enumerate}
\item Planification
\begin{enumerate}
\item Allocation of temporary arrays;
\item Computation of the partitioning indices;
\item Computation of the recurrence coefficients;
\item Planification of the in-place DCT-I and DST-I;
\item Computation of the modified weights and orthonormality constants (inverse only);
\end{enumerate}
\item Execution
\begin{enumerate}
\item Computation of the diagonal matrices ${\bf D}_{u_m(\tc)}$, ${\bf D}_{v_m(\tc)}$, and ${\bf D}_{C_{n,m}^{\alpha,\beta}}$;
\item Application of the DCT-I and DST-I; and,
\item Execution of the recurrence relations.
\end{enumerate}
\end{enumerate}

Since part 1 is only dependent on the degree and the Jacobi parameters, it is reusable. Therefore, results of part 1 are stored in an object called a {\tt ChebyshevJacobiPlan}. Analogous to FFTW, applying the {\tt ChebyshevJacobiPlan} to a vector results in execution of part 2. While it is beneficial to divide the computation like so, the construction of a {\tt ChebyshevJacobiPlan} is not orders of magnitude larger than the execution, as is the case for other schemes using hierarchical or other complex data structures; our numerical experiments suggest an approximate gain on the order of $10\%$. However, the reduction of memory allocation alone could be important in memory-sensitive applications.

\subsection{Computational issues}

Consider the Stirling series for the gamma function~\citep[\S 5.11.10]{Olver-et-al-NIST-10} on $z\in\mathbb{R}^+$:
\begin{equation}
\Gamma(z) = \sqrt{2\pi}z^{z-\frac{1}{2}} e^{-z}\left(S_N(z)+R_N(z)\right),\qquad S_N(z) = \sum_{n=0}^{N-1} \dfrac{a_n}{z^n},\quad R_N(z) \le \dfrac{(1+\zeta(N))\Gamma(N)}{(2\pi)^{N+1}z^N}.
\end{equation}
The sequence $\{a_n\}_{n\ge0}$ is defined by the ratio of sequences A001163 and A001164 of~\citet{Sloane-OEIS}, and $\zeta$ is the Riemann zeta function~\citep[\S 25]{Olver-et-al-NIST-10}. Table~\ref{table:Stirling} shows the necessary and sufficient number of terms required of the Stirling series such that $\tfrac{R_N(z)}{S_N(z)} < \tfrac{\varepsilon}{20} \approx 1.1102\times10^{-17}$. Taking rounding errors into account, the effect is a relative error below machine precision $\varepsilon \approx 2.2204\times10^{-16}$ in double precision arithmetic.
\begin{table}[htbp]
\begin{center}
\caption{Number of terms such that $\tfrac{R_N(z)}{S_N(z)} < \tfrac{\varepsilon}{20} \approx 1.1102\times10^{-17}$ in double precision arithmetic.}
\label{table:Stirling}
\begin{tabular}{c|ccccccc}
\hline
$z\ge$ & 3275 & 591 & 196 & 92 & 53 & 35 & 26\\
$N$ & 4 & 5 & 6 & 7 & 8 & 9 & 10\\
\hline
$z\ge$ & 20 & 17 & 14 & 12 & 11 & 10 & 9\\
$N$ & 11 & 12 & 13 & 14 & 15 & 16 & 17\\
\hline
\end{tabular}
\end{center}
\end{table}%

Define $S_{\varepsilon}(z): [9,\infty)\to\mathbb{R}$ by the truncated Stirling series $S_N(z)$ with necessary and sufficient $N$ for relative error below $\varepsilon$, as determined by Table~\ref{table:Stirling}.

The coefficients $C_{n,m}^{\alpha,\beta}$ of~\eqref{eq:Cnmab} can be stably computed by forward recurrence in $n$ and $m$:
\begin{equation}
C_{n,m}^{\alpha,\beta} = \dfrac{(n+\alpha)(n+\beta)C_{n-1,m}^{\alpha,\beta}}{(n+\frac{\alpha+\beta+m+1}{2})(n+\frac{\alpha+\beta+m}{2})},\quad{\rm and}\quad
C_{n,m}^{\alpha,\beta} = \dfrac{C_{n,m-1}^{\alpha,\beta}}{2(2n+\alpha+\beta+m+1)}.\label{eq:Cnmabrec}
\end{equation}
However, to determine the indices $i_k^1$ and $i_k^2$ for the partitioning of the matrix $\PN{\alpha,\beta}$, use of an asymptotic formula is more efficient. Here, we adapt the approach of~\citep[\S 3.3.1]{Hale-Townsend-35-A652-13} with suitable modifications. In terms of $S_\varepsilon(z)$ defined above, the coefficients $C_{n,m}^{\alpha,\beta}$ can be expressed as:
\begin{align}
C_{n,m}^{\alpha,\beta} 
& = \frac{2^{2n-m+\alpha+\beta+1}\sqrt{2\pi} (n+\alpha+1)^{n+\alpha+\frac{1}{2}} e^{-n-\alpha-1} (n+\beta+1)^{n+\beta+\frac{1}{2}}e^{-n-\beta-1}}{\pi (2n+m+\alpha+\beta+2)^{2n+m+\alpha+\beta+\frac{3}{2}} e^{-2n-m-\alpha-\beta-2}}\nonumber\\
& \quad\times\dfrac{S_\varepsilon(n+\alpha+1)S_\varepsilon(n+\beta+1)}{S_\varepsilon(2n+m+\alpha+\beta+2)},\\
& = \frac{e^m}{4^m\sqrt{\pi}} \left(1+\frac{\alpha-\beta-m}{2n+\alpha+\beta+m+2}\right)^{n+\alpha+\frac{1}{2}} \left(1+\frac{\beta-\alpha-m}{2n+\alpha+\beta+m+2}\right)^{n+\beta+\frac{1}{2}}\nonumber\\
& \quad\times \frac{1}{n^{m+\frac{1}{2}}(1+\frac{\alpha+\beta+m+2}{2n})^{m+\frac{1}{2}}} \dfrac{S_\varepsilon(n+\alpha+1)S_\varepsilon(n+\beta+1)}{S_\varepsilon(2n+m+\alpha+\beta+2)}.\label{eq:Cnmabasy}
\end{align}
In~\eqref{eq:Cnmabasy}, the terms resembling $(1+x)^y$ can be computed stably and efficiently by $\exp(y\logop x)$, where $\logop$ calls the natural logarithm $\log(1+x)$ for large arguments and its Taylor series for small arguments. So long as $n+\min\{\alpha,\beta\}\ge8$, the asymptotic formula~\eqref{eq:Cnmabasy} for the coefficients $C_{n,m}^{\alpha,\beta}$ can be used for a fast and stable numerical evaluation, and the downward recurrence of~\eqref{eq:Cnmabrec} supplies $C_{n,m}^{\alpha,\beta}$ for the handful of remaining values.

To compute the orthonormality constants $\mathscr{A}_n^{\alpha,\beta}$ of~\eqref{eq:Anab}, the asymptotic expansion derived by~\citet{Buhring-24-505-00} can be used. However, a remainder estimate is not reported and instead we use the same technique as for the computation of the coefficients $C_{n,m}^{\alpha,\beta}$:
\begin{align}
\mathscr{A}_n^{\alpha,\beta} & = \dfrac{2^{\alpha+\beta+1}}{2n+\alpha+\beta+1}\dfrac{(n+\alpha+1)^{n+\alpha+\frac{1}{2}}(n+\beta+1)^{n+\beta+\frac{1}{2}}}{(n+\alpha+\beta+1)^{n+\alpha+\beta+\frac{1}{2}}(n+1)^{n+\frac{1}{2}}}\dfrac{S_\varepsilon(n+\alpha+1)S_\varepsilon(n+\beta+1)}{S_\varepsilon(n+\alpha+\beta+1)S_\varepsilon(n+1)},\\
& = \dfrac{2^{\alpha+\beta+1}}{2n+\alpha+\beta+1}\left(1-\dfrac{\beta}{n+\alpha+\beta+1}\right)^{\frac{n}{2}+\alpha+\frac{1}{4}}\left(1-\dfrac{\alpha}{n+\alpha+\beta+1}\right)^{\frac{n}{2}+\beta+\frac{1}{4}}\nonumber\\
& \quad\times \left(1+\dfrac{\alpha}{n+1}\right)^{\frac{n}{2}+\frac{1}{4}}\left(1+\dfrac{\beta}{n+1}\right)^{\frac{n}{2}+\frac{1}{4}} \dfrac{S_\varepsilon(n+\alpha+1)S_\varepsilon(n+\beta+1)}{S_\varepsilon(n+\alpha+\beta+1)S_\varepsilon(n+1)}.\label{eq:Anabasy}
\end{align}
Similar to~\eqref{eq:Cnmabasy},~\eqref{eq:Anabasy} can be computed stably and efficiently for $n+\min\{\alpha,\beta,\alpha+\beta,0\}\ge8$. Note as well the symmetry in both~\eqref{eq:Cnmabasy} and~\eqref{eq:Anabasy} upon the substitution $\alpha\leftrightarrow\beta$.

\subsection{Reinsch's modification of forward orthogonal polynomial recurrence and the Clenshaw--Smith algorithm}

In order to evaluate $\PN{\alpha,\beta}^{\rm REC}$ and its transpose, recurrence relations are required. Here, we review recurrence relations for orthogonal polynomials and derive new relations for stabilized evaluation near the boundary of the interval of orthogonality for Jacobi polynomials.

Let an orthogonal polynomial sequence $\pi_n(x)$ be defined by the three-term recurrence relation~\citep[\S 18.9.1]{Olver-et-al-NIST-10}:
\begin{equation}\label{eq:OPrec}
\pi_{n+1}(x) = (A_nx+B_n)\pi_n(x) -C_n\pi_{n-1}(x),\qquad \pi_{-1}(x) = 0,\quad \pi_0(x) = 1.
\end{equation}
The Clenshaw--Smith algorithm writes the sum:
\begin{equation}
p_N(x) = \sum_{n=0}^N c_n \pi_n(x),
\end{equation}
via an inhomogeneous recurrence relation involving the adjoint of~\eqref{eq:OPrec} as follows:
\begin{algorithm}[\citet{Clenshaw-9-118-55,Smith-19-33-65}]~
\begin{enumerate}
\item Set:
\begin{equation}
u_{N+1}(x) = u_{N+2}(x) = 0.
\end{equation}
\item For $n=N,N-1,\ldots,0$:
\begin{equation}
u_n(x) = (A_nx+B_n)u_{n+1}(x) - C_{n+1}u_{n+2}(x) + c_n.
\end{equation}
\item Then:
\begin{equation}
p_N(x) = u_0(x).
\end{equation}
\end{enumerate}
\end{algorithm}

After Clenshaw's original error analysis, it was~\citet{Gentleman-12-160-69} who first drew attention to the susceptibility of larger rounding errors near the ends of the interval $[-1,1]$. In Gentleman's paper, Reinsch proposed (unpublished) a stabilizing modification, with the error analysis of the modification performed by~\citet{Oliver-20-379-77}.~\citet{Levrie-Piessens-74-85} derive Reinsch's modification of the Clenshaw--Smith algorithm for Legendre, ultraspherical, and Laguerre polynomials. They also derive Reinsch's modification to the forward orthogonal polynomial recurrence~\eqref{eq:OPrec} for Chebyshev, Legendre, ultraspherical, Jacobi, and Laguerre polynomials. Here, we review Reinsch's modification with more general notation than that of~\citet{Levrie-Piessens-74-85} and extend Reinsch's modified Clenshaw--Smith algorithm to Jacobi polynomials.

Formally, we define the ratio:
\begin{equation}\label{eq:rnfx0}
r_n^f(x) := \dfrac{\pi_{n+1}(x)}{\pi_n(x)},\quad{\rm for}\quad n\ge0,
\end{equation}
such that at the point $x_0$:
\begin{equation}\label{eq:rn}
r_n^f(x_0) = A_nx_0+B_n - C_nr_{n-1}^f(x_0)^{-1},
\end{equation}
or isolating for $B_n$:
\begin{equation}
B_n = r_n^f(x_0)+C_nr_{n-1}^f(x_0)^{-1}-A_nx_0.
\end{equation}
Substituting this relationship for $B_n$ into the forward recurrence~\eqref{eq:OPrec}, we obtain the modified version:
\begin{algorithm}~
\begin{enumerate}
\item Set:
\begin{equation}
\pi_0(x) =1,\qquad d_0(x) = 0.
\end{equation}
\item For $n\ge0$:
\begin{align}
d_{n+1}(x) & = \left( A_n(x-x_0)\pi_n(x) + C_nd_n(x) \right)r_n^f(x_0)^{-1},\\
\pi_{n+1}(x) & = \left(\pi_n(x) + d_{n+1}(x)\right)r_n^f(x_0).
\end{align}
\end{enumerate}
\end{algorithm}

Consider the homogeneous adjoint three-term recurrence:
\begin{equation}\label{eq:adjointvn}
v_n(x) = (A_nx+B_n)v_{n+1}(x) - C_{n+1}v_{n+2}(x),\qquad v_0(x) = 0,\quad v_1(x) = 1.
\end{equation}
Formally, we define the ratio:
\begin{equation}\label{eq:rnbx0}
r_n^b(x) := \dfrac{v_{n+1}(x)}{v_n(x)},\quad{\rm for}\quad n>0,
\end{equation}
such that at the point $x_0$:
\begin{equation}\label{eq:adjointrn}
r_n^b(x_0)^{-1} = A_nx_0+B_n - C_{n+1}r_{n+1}^b(x_0),
\end{equation}
or isolating for $B_n$:
\begin{equation}
B_n = r_n^b(x_0)^{-1} + C_{n+1}r_{n+1}^b(x_0) - A_nx_0.
\end{equation}
Substituting this relationship for $B_n$ into the Clenshaw--Smith algorithm, we obtain the modified version:
\begin{algorithm}~
\begin{enumerate}
\item Set:
\begin{equation}
u_{N+1}(x) = d_{N+1}(x) = 0.
\end{equation}
\item For $n=N,N-1,\ldots,1$:
\begin{align}
d_n(x) & = \left( A_n(x-x_0)u_{n+1}(x) + C_{n+1}d_{n+1}(x) + c_n \right)r_n^b(x_0),\\
u_n(x) & = \left(u_{n+1}(x) + d_n(x)\right)r_n^b(x_0)^{-1}.
\end{align}
\item Then:
\begin{equation}
p_N(x) = A_0(x-x_0)u_1(x) + C_1d_1(x) + c_0.
\end{equation}
\end{enumerate}
\end{algorithm}

The stability of the modified forward recurrence and the modified Clenshaw--Smith algorithm near $x_0$ is derived from the geometric damping induced by $x-x_0$ and the avoidance of cancellation errors. However, the na\"ive implementation of the two-term recurrence relations for the ratios~\eqref{eq:rn} and~\eqref{eq:adjointrn} contains precisely the cancellation errors we were hoping to avoid. Therefore, to complete the stable implementation of the scheme, we require stable evaluation of the ratios $r^f(x_0)$ and $r^b(x_0)$.

For Jacobi polynomials, due to~\eqref{eq:Pnendpts} and the two-term recurrence of binomials, the ratios $r_n^f(\pm1)$ defined by~\eqref{eq:rnfx0} are trivial:
\begin{equation}
r_n^f(1) = \dfrac{n+\alpha+1}{n+1},\quad{\rm and}\quad r_n^f(-1) = -\frac{n+\beta+1}{n+1}.
\end{equation}
Fortunately, we can also prove the following:
\begin{lemma}
For Jacobi polynomials, the ratios $r_n^b(\pm1)$ defined by~\eqref{eq:rnbx0} are:
\begin{equation}
r_n^b(1) = \dfrac{n+1}{n}\dfrac{(\alpha+\beta+n+1)(\alpha+\beta+2n)}{(n+\beta)(\alpha+\beta+2n+2)},\quad{\rm and}\quad r_n^b(-1) = -\dfrac{n+1}{n}\dfrac{(\alpha+\beta+n+1)(\alpha+\beta+2n)}{(n+\alpha)(\alpha+\beta+2n+2)}.
\end{equation}
\end{lemma}
\begin{proof}
One need only insert the ratios into the relationship~\eqref{eq:adjointrn}.
\end{proof}

Figure~\ref{fig:recurrence} shows the relative error in evaluating $P_{10,000}^{(0,0)}(\cos\theta)$ at $10,001$ equally spaced angles using the six described algorithms. In Figure~\ref{fig:recurrence}, the terms $x\pm1$ are computed accurately with the trigonometric identities $x+1=2\cos^2(\tfrac{\theta}{2})$ and $x-1 = -2\sin^2(\tfrac{\theta}{2})$. While variations in $\alpha$ and $\beta$ will change the accuracy of all six recurrence relations, practically, we take the unmodified algorithms to be more accurate in $\frac{\pi}{4} < \theta < \frac{3\pi}{4}$, and the modifications otherwise as the perturbations in the breakpoints are asymptotically of lower order as $N\to\infty$.

\begin{figure}[htpb]
\begin{center}
\begin{tabular}{cc}
\hspace*{-0.5cm}\includegraphics[width=0.5\textwidth]{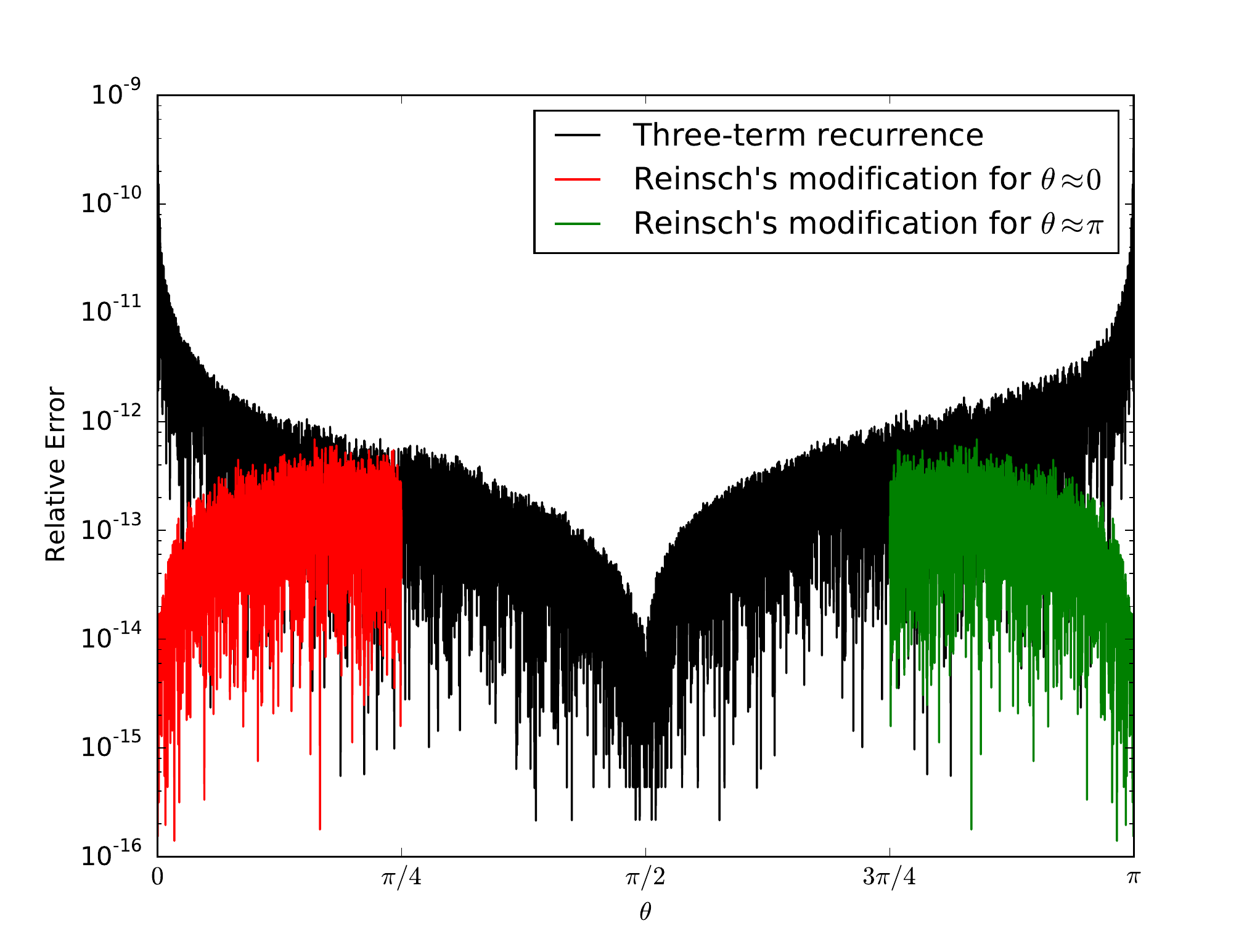}&
\hspace*{-0.5cm}\includegraphics[width=0.5\textwidth]{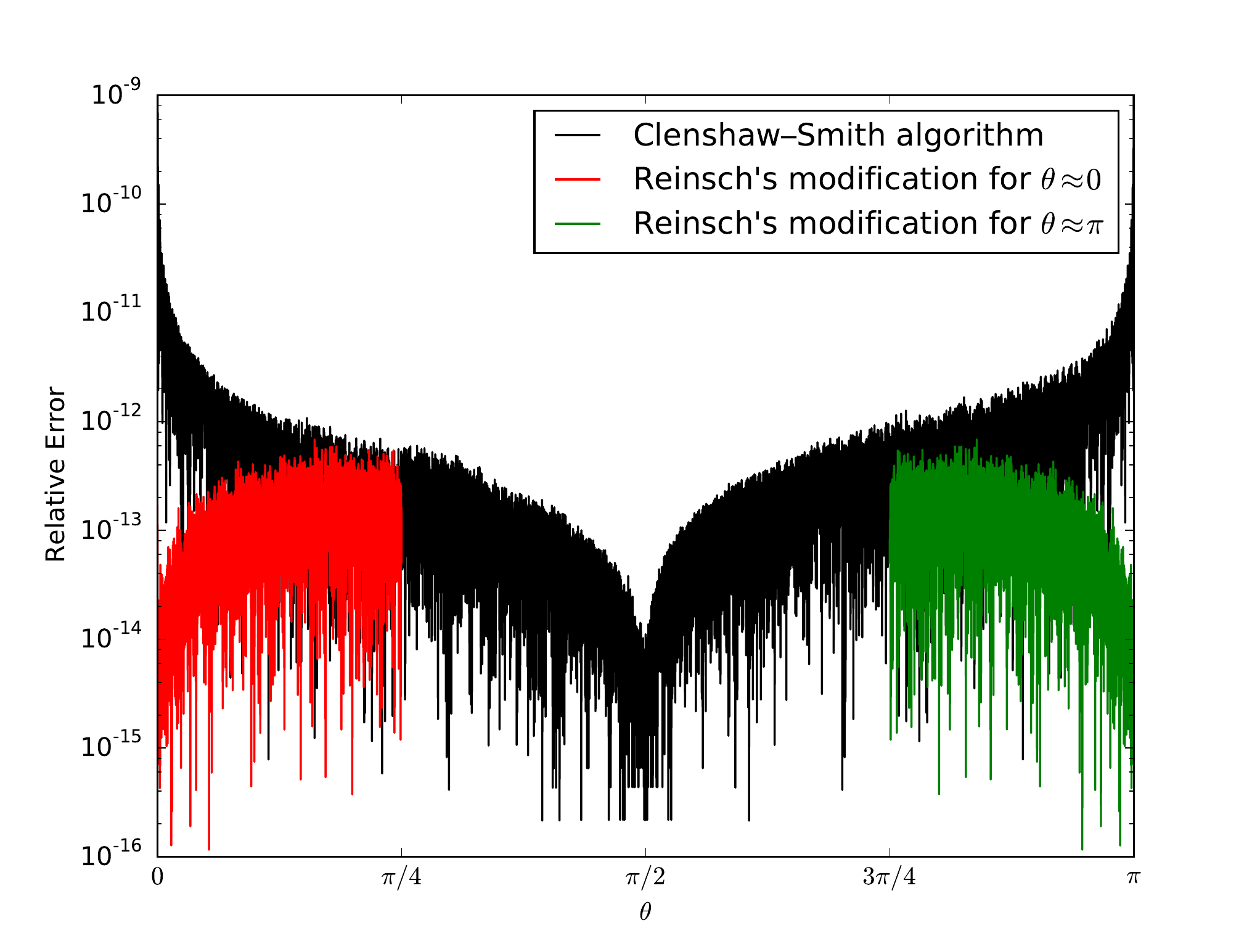}\\
\end{tabular}
\caption{The relative error in evaluating $P_{10,000}^{(0,0)}(\cos\theta)$ at $10,001$ equally spaced angles using: left, the three-term recurrence relation and Reinsch's modification for the ends of the interval $[0,\pi]$; right, the Clenshaw--Smith algorithm and Reinsch's modification for the ends of the interval $[0,\pi]$.}
\label{fig:recurrence}
\end{center}
\end{figure}

\section{Numerical discussion \& outlook}

In principle the connection coefficients are able to provide reference solutions for the maximum absolute error. But in practice, the na\"ive algorithm's quadratic complexity limits the applicability to below about $N=10^4$. Therefore, in Figure~\ref{fig:error}, we plot the maximum absolute error in transforming Chebyshev expansion coefficients to Jacobi expansion coefficients and back for coefficients simulating an irregular function and for coefficients simulating a continuous function. Error is similar for the forward--inverse composition. Figure~\ref{fig:timings} shows the execution time of the forward and inverse transforms in line with the predicted asymptotic complexity ${\cal O}(N\log^2N/\log\log N)$. Our implementation~\citep[\tt FastTransforms.jl]{Slevinsky-GitHub-FastTransforms} in the {\sc Julia} programming language is freely available online.

Composition of the forward and inverse transforms allows for the transform between expansions in Jacobi polynomials of differing parameters. As well, use of a Nonuniform Discrete Cosine Transform (NDCT) (e.g.~\citet{Hale-Townsend-16}) could allow for fast evaluation at the Gauss--Jacobi nodes. However, an efficient NDCT requires points to be close to the Chebyshev points of the first kind, and the inequalities on the zeros of the Jacobi polynomials~\citep[\S 18.16]{Olver-et-al-NIST-10} seem to be overestimates. The performance of an NDCT may be better in practice than can be currently estimated theoretically.

One potential area of application is the extension of the fast and well-conditioned spectral method for solving singular integral equations of~\citet{Slevinsky-Olver-15} to polygonal boundaries. Elliptic partial differential equations have angle-dependent algebraic singularities in the densities on polygonal boundaries. It is conjectured that working in the more exotic bases of Jacobi polynomials and Jacobi functions of the second kind can lead to banded representations of singular integral operators defined on polygonal boundaries.

Since the integer-valued increments are required for Jacobi parameters beyond $(\alpha,\beta)\in(-\tfrac{1}{2},\tfrac{1}{2}]^2$, the method proposed and analyzed here cannot be used for exceedingly large parameters. This is consistent with nonuniformity of Hahn's asymptotics~\eqref{eq:Jacasy} in $\alpha$ and $\beta$. Therefore, this Chebyshev--Jacobi transform cannot be used for a fast spherical harmonics transform. There are certain parameter r\'egimes where the complexity can be reduced. These are detailed in Appendix~\ref{appendix:EdgeCases}.

\begin{figure}[htpb]
\begin{center}
\begin{tabular}{cc}
\hspace*{-0.5cm}\includegraphics[width=0.5\textwidth]{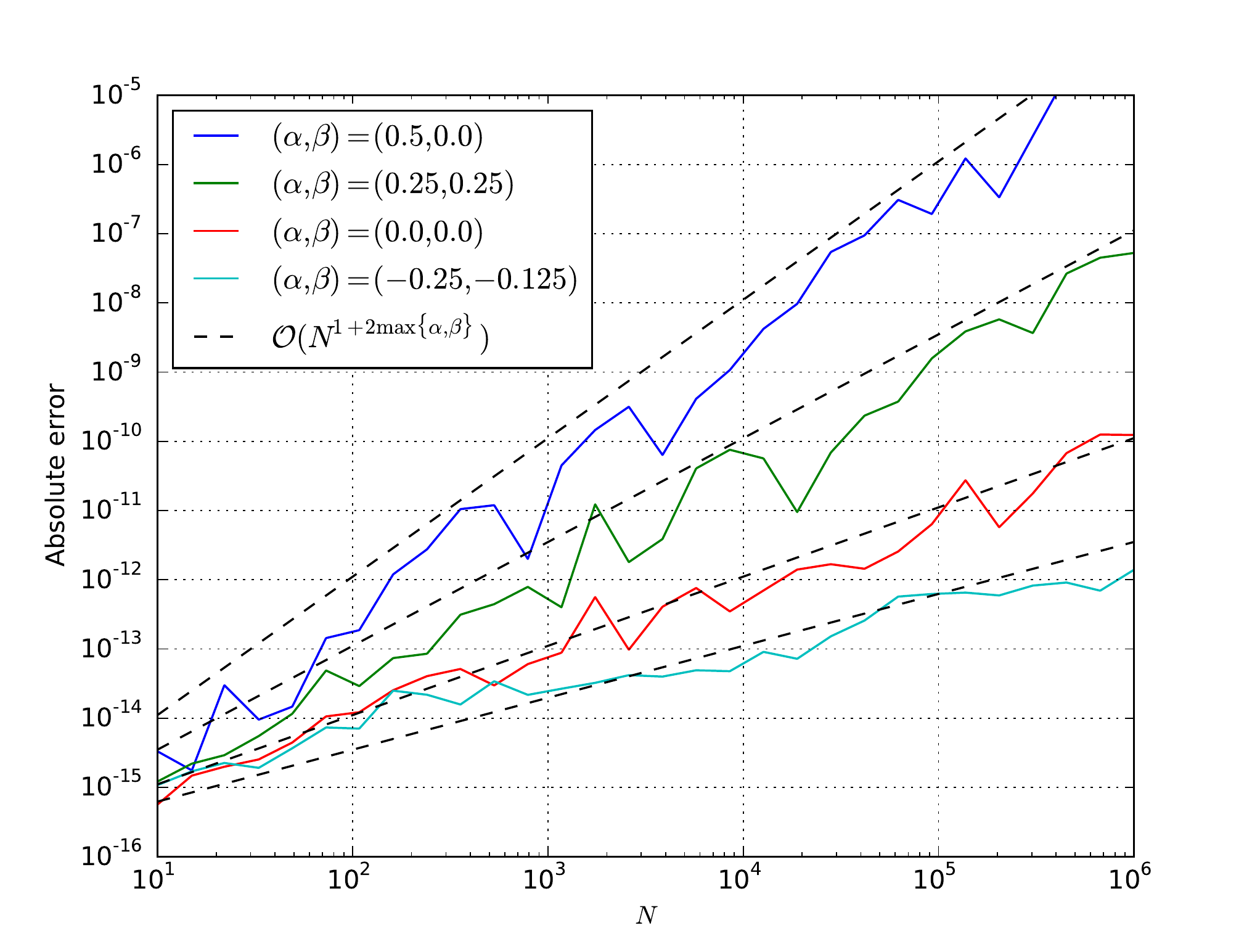}&
\hspace*{-0.5cm}\includegraphics[width=0.5\textwidth]{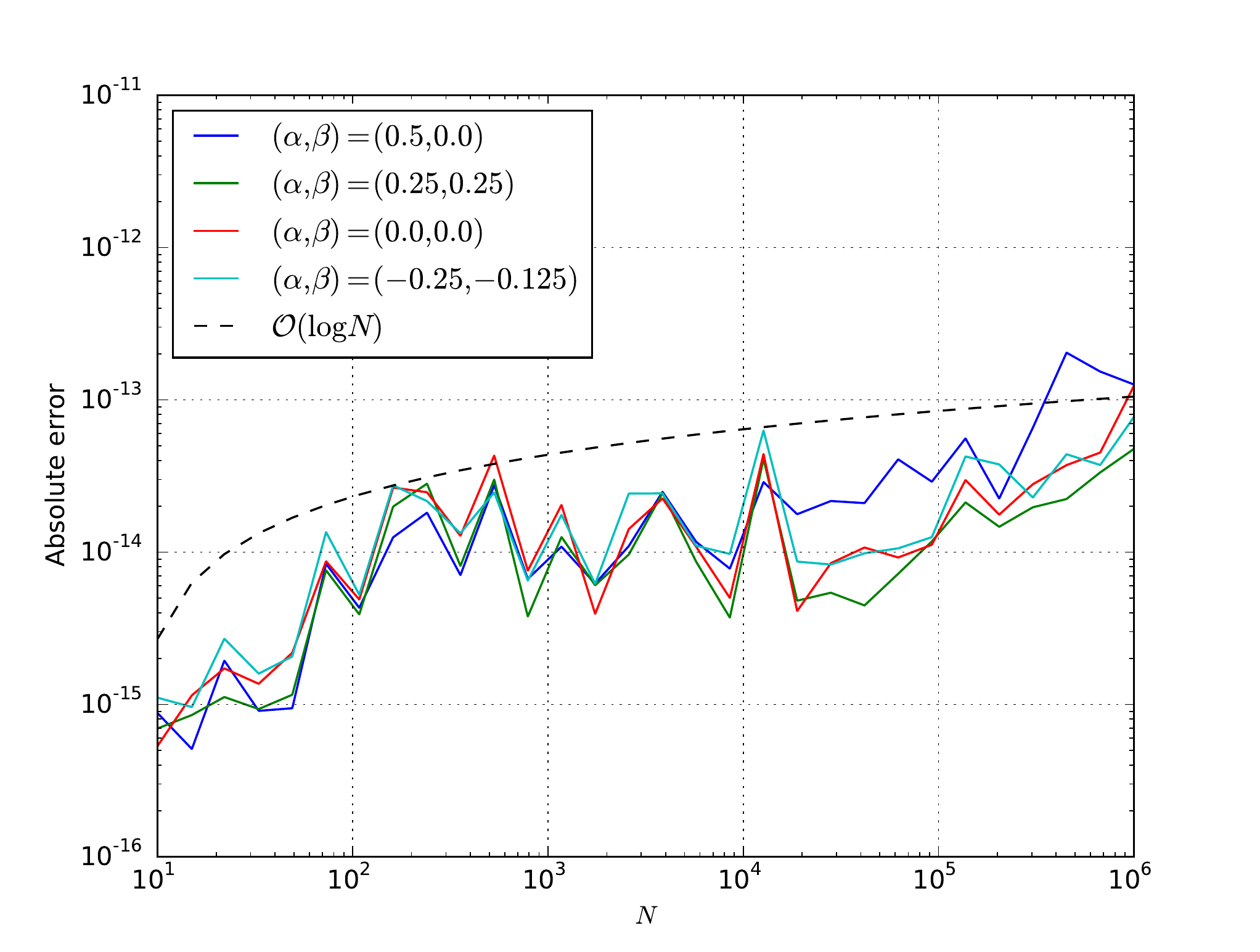}\\
\end{tabular}
\caption{Absolute error of the inverse Chebyshev--Jacobi transform of the forward Chebyshev--Jacobi transform. Left: for coefficients simulating an irregular function $[{\bf c}_N]_n \sim U(0,1)$. Right: for coefficients simulating a continuous function $[{\bf c}_N]_n \sim U(-1,1)n^{-2}$. In both plots, the numbers labeling the solid lines refer to different Jacobi parameters and the dashed black lines are asymptotic estimates on the error based on the error analyses. The results are an average over $10$ executions.}
\label{fig:error}
\end{center}
\end{figure}

\begin{figure}[htpb]
\begin{center}
\begin{tabular}{cc}
\hspace*{-0.5cm}\includegraphics[width=0.5\textwidth]{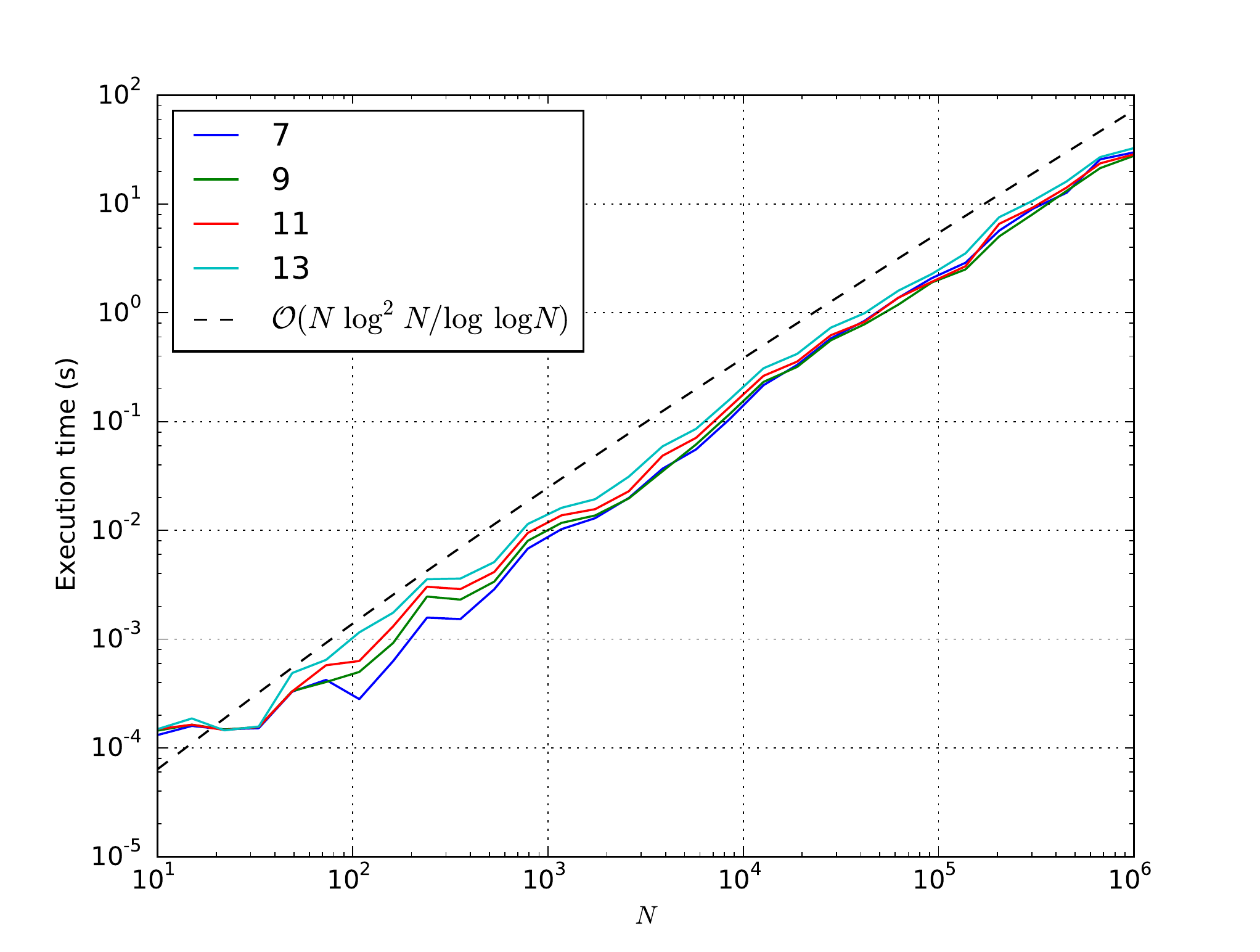}&
\hspace*{-0.5cm}\includegraphics[width=0.5\textwidth]{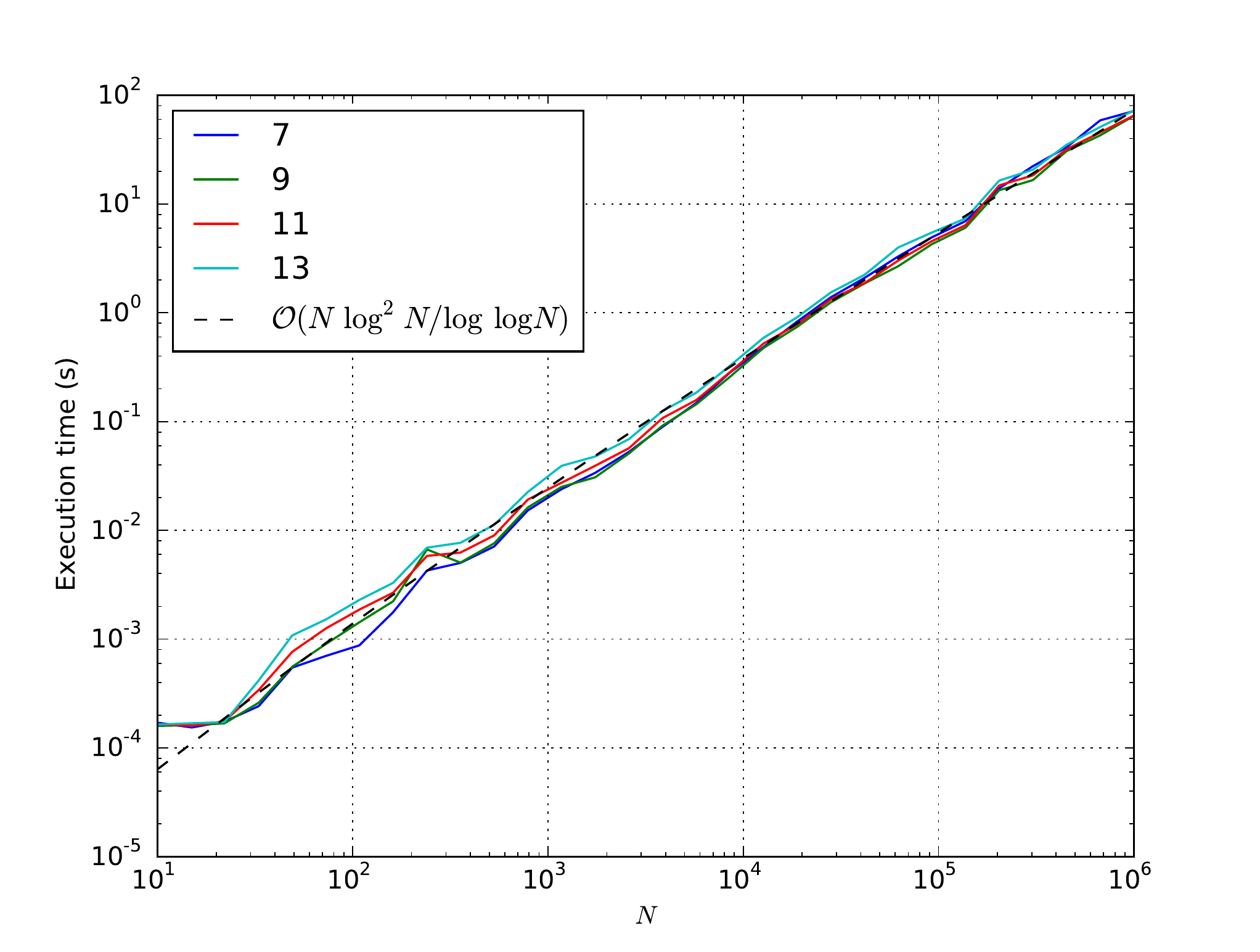}\\
\end{tabular}
\caption{Execution time for the Chebyshev--Jacobi transform. Left: the forward Chebyshev--Jacobi transform. Right: the inverse Chebyshev--Jacobi transform. In both plots, the numbers labeling the solid lines refer to the parameter $M$ and the dashed black line is the same, showing that the inverse transform takes about twice the time. The results are an average over $10$ executions performed with uniformly distributed parameters $\alpha\sim U(-\tfrac{1}{2},\tfrac{1}{2})$ and $\beta\sim U(-\tfrac{1}{2},\tfrac{1}{2})$.}
\label{fig:timings}
\end{center}
\end{figure}

\section*{Acknowledgments}

This paper is dedicated to the celebration of Nick Trefethen on his $60^{\rm th}$ birthday and his inspirational contributions to numerical analysis. I also acknowledge the generous support of the Natural Sciences and Engineering Research Council of Canada.

\bibliography{/Users/Mikael/Bibliography/Mik}

\begin{thebibliography}{}

\bibitem[Abramowitz \& Stegun(1965)Abramowitz \& Stegun]{Abramowitz-Stegun-65}
{\sc Abramowitz, M. \& Stegun, I.~A.} (1965)
\newblock {\em Handbook of Mathematical Functions\/}.
\newblock New York: Dover.

\bibitem[Alpert \& Rokhlin(1991)Alpert \& Rokhlin]{Alpert-Rokhlin-12-158-91}
{\sc Alpert, B.~K. \& Rokhlin, V.} (1991)
\newblock A fast algorithm for the evaluation of {L}egendre expansions.
\newblock {\em SIAM J. Sci. Stat. Comput.}, {\bf 12}, 158--179.

\bibitem[Andrews {\em et~al.}(1998)Andrews, Askey, \&
  Roy]{Andrews-Askey-Roy-98}
{\sc Andrews, G.~E., Askey, R. \& Roy, R.} (1998)
\newblock {\em Special Functions\/}.
\newblock Cambridge University Press.

\bibitem[B{\"u}hring(2000)B{\"u}hring]{Buhring-24-505-00}
{\sc B{\"u}hring, W.} (2000)
\newblock An asymptotic expansion for a ratio of products of gamma functions.
\newblock {\em Internat. J. Math. \& Math. Sci.}, {\bf 24}, 505--510.

\bibitem[Cantero \& Iserles(2012)Cantero \& Iserles]{Cantero-Iserles-50-307-12}
{\sc Cantero, M.~J. \& Iserles, A.} (2012)
\newblock On rapid computation of expansions in ultraspherical polynomials.
\newblock {\em SIAM J. Numer. Anal.}, {\bf 50}, 307--327.

\bibitem[Clenshaw(1955)Clenshaw]{Clenshaw-9-118-55}
{\sc Clenshaw, C.~W.} (1955)
\newblock A note on the summation of {C}hebyshev series.
\newblock {\em Math. Comp.}, {\bf 9}, 118--120.

\bibitem[Frigo \& Johnson(2005)Frigo \& Johnson]{Frigo-Johnson-93-216-05}
{\sc Frigo, M. \& Johnson, S.~G.} (2005)
\newblock The design and implementation of {FFTW3}.
\newblock {\em Proc. IEEE\/}, {\bf 93}, 216--231.

\bibitem[Gentleman(1969)Gentleman]{Gentleman-12-160-69}
{\sc Gentleman, W.~M.} (1969)
\newblock An error analysis of {G}oertzel's ({W}att's) method for computing
  {F}ourier coefficients.
\newblock {\em Comput. J.}, {\bf 12}, 160--164.

\bibitem[Graham {\em et~al.}(1989)Graham, Knuth, \&
  Patashnik]{Graham-Knuth-Patashnik-89}
{\sc Graham, R.~L., Knuth, D.~E. \& Patashnik, O.} (1989)
\newblock {\em Concrete Mathematics, A Foundation for Computer Science\/},
  second edn.
\newblock Addison-Wesley.

\bibitem[Hahn(1980)Hahn]{Hahn-171-201-80}
{\sc Hahn, E.} (1980)
\newblock Asymptotik bei {J}acobi-polynomen und {J}acobi-funktionen.
\newblock {\em Math. Z.}, {\bf 171}, 201--226.

\bibitem[Hale \& Townsend(2013)Hale \& Townsend]{Hale-Townsend-35-A652-13}
{\sc Hale, N. \& Townsend, A.} (2013)
\newblock Fast and accurate computation of {G}auss--{L}egendre and
  {G}auss--{J}acobi quadrature nodes and weights.
\newblock {\em SIAM J. Sci. Comput.}, {\bf 35}, A652--A674.

\bibitem[Hale \& Townsend(2014)Hale \& Townsend]{Hale-Townsend-36-A148-14}
{\sc Hale, N. \& Townsend, A.} (2014)
\newblock A fast, simple, and stable {C}hebyshev--{L}egendre transform using an
  asymptotic formula.
\newblock {\em SIAM J. Sci. Comput.}, {\bf 36}, A148--A167.

\bibitem[Hale \& Townsend(2016)Hale \& Townsend]{Hale-Townsend-16}
{\sc Hale, N. \& Townsend, A.} (2016)
\newblock A fast {FFT}-based discrete {L}egendre transform.
\newblock {\em IMA J. Numer. Anal.}

\bibitem[Keiner(2009)Keiner]{Keiner-31-2151-09}
{\sc Keiner, J.} (2009)
\newblock Computing with expansions in {G}egenbauer polynomials.
\newblock {\em SIAM J. Sci. Comput.}, {\bf 31}, 2151--2171.

\bibitem[Levrie \& Piessens(1985)Levrie \& Piessens]{Levrie-Piessens-74-85}
{\sc Levrie, P. \& Piessens, R.} (1985)
\newblock A note on the evaluation of orthogonal polynomials using recurrence
  relations.
\newblock {\em Technical Report}~74.
\newblock Katholieke Universiteit Leuven.

\bibitem[Li \& Shen(2010)Li \& Shen]{Li-Shen-79-1621-10}
{\sc Li, H. \& Shen, J.} (2010)
\newblock Optimal error estimates in {J}acobi-weighted {S}obolev spaces for
  polynomial approximations on the triangle.
\newblock {\em Math. Comp.}, {\bf 79}, 1621--1646.

\bibitem[Mason \& Handscomb(2002)Mason \& Handscomb]{Mason-Handscomb-02}
{\sc Mason, J.~C. \& Handscomb, D.~C.} (2002)
\newblock {\em Chebyshev Polynomials\/}.
\newblock CRC Press.

\bibitem[Mori {\em et~al.}(1999)Mori, Suda, \&
  Sugihara]{Mori-Suda-Sugihara-40-3612-99}
{\sc Mori, A., Suda, R. \& Sugihara, M.} (1999)
\newblock An improvement on {O}rszag's fast algorithm for {L}egendre polynomial
  transform.
\newblock {\em Trans. Info. Process. Soc. Japan\/}, {\bf 40}, 3612--3615.

\bibitem[Oliver(1968)Oliver]{Oliver-11-349-68}
{\sc Oliver, J.} (1968)
\newblock The numerical solution of linear recurrence relations.
\newblock {\em Numer. Math.}, {\bf 11}, 349--360.

\bibitem[Oliver(1977)Oliver]{Oliver-20-379-77}
{\sc Oliver, J.} (1977)
\newblock An error analysis of the modified {C}lenshaw method for evaluating
  {C}hebyshev and {F}ourier series.
\newblock {\em IMA J. App. Math.}, {\bf 20}, 379--391.

\bibitem[Olver {\em et~al.}(2010)Olver, Lozier, Boisvert, \&
  Clark]{Olver-et-al-NIST-10}
{\sc Olver, F. W.~J., Lozier, D.~W., Boisvert, R.~F. \& Clark, C.~W.} (eds)
  (2010)
\newblock {\em NIST Handbook of Mathematical Functions\/}.
\newblock Cambridge U. P.

\bibitem[Orszag(1986)Orszag]{Orszag-13-86}
{\sc Orszag, S.~A.} (1986)
\newblock Fast eigenfunction transforms.
\newblock {\em Science and Computers\/}.
\newblock New York: Academic Press, pp. 13--30.

\bibitem[Piessens(1987)Piessens]{Piessens-87}
{\sc Piessens, R.} (1987)
\newblock {\em Numerical Integration\/}, vol.
\newblock 203.
\newblock Springer Netherlands, chapter 2, Modified Clenshaw--Curtis
  Integration and Applications to Numerical Computation of Integral Transforms.

\bibitem[Rainville(1960)Rainville]{Rainville-60}
{\sc Rainville, E.} (1960)
\newblock {\em Special Functions\/}.
\newblock MacMillan.

\bibitem[Slevinsky(2016)Slevinsky]{Slevinsky-GitHub-FastTransforms}
{\sc Slevinsky, R.~M.} (2016)
\newblock {\tt https://github.com/MikaelSlevinsky/FastTransforms.jl}.

\bibitem[Slevinsky \& Olver(2015)Slevinsky \& Olver]{Slevinsky-Olver-15}
{\sc Slevinsky, R.~M. \& Olver, S.} (2015)
\newblock A fast and well-conditioned spectral method for singular integral
  equations.
\newblock arXiv:1507.00596.

\bibitem[Sloane(2016)Sloane]{Sloane-OEIS}
{\sc Sloane, N. J.~A.} (2016)
\newblock The {O}n-{L}ine {E}ncyclopedia of {I}nteger {S}equences.
\newblock {\tt http://oeis.org}.

\bibitem[Smith(1965)Smith]{Smith-19-33-65}
{\sc Smith, F.~J.} (1965)
\newblock An algorithm for summing orthogonal polynomial series and their
  derivatives with applications to curve-fitting and interpolation.
\newblock {\em Math. Comp.}, {\bf 19}, 33--36.

\bibitem[Sommariva(2013)Sommariva]{Sommariva-65-682-13}
{\sc Sommariva, A.} (2013)
\newblock Fast construction of {F}ej\'er and {C}lenshaw--{C}urtis rules for
  general weight functions.
\newblock {\em Comp. Math. Appl.}, {\bf 65}, 682--693.

\bibitem[Trefethen(2012)Trefethen]{Trefethen-12}
{\sc Trefethen, L.~N.} (2012)
\newblock {\em Approximation Theory and Approximation Practice\/}.
\newblock SIAM.

\bibitem[Waldvogel(2003)Waldvogel]{Waldvogel-43-001-03}
{\sc Waldvogel, J.} (2003)
\newblock Fast construction of the {F}ej\'er and {C}lenshaw--{C}urtis
  quadrature rules.
\newblock {\em BIT Numer. Math.}, {\bf 43}, 001--018.

\bibitem[Wang \& Huybrechs(2014)Wang \& Huybrechs]{Wang-Huybrechs-14a}
{\sc Wang, H. \& Huybrechs, D.} (2014)
\newblock Fast and accurate computation of {J}acobi expansion coefficients of
  analytic functions.
\newblock arXiv:1404.2463v1.

\bibitem[Wimp {\em et~al.}(1997)Wimp, McCabe, \&
  Connor]{Wimp-McCabe-Connor-82-447-97}
{\sc Wimp, J., McCabe, P. \& Connor, J. N.~L.} (1997)
\newblock Computation of {J}acobi functions of the second kind for use in
  nearside--farside scattering theory.
\newblock {\em J. Comp. Appl. Math.}, {\bf 82}, 447--464.

\bibitem[Xiang {\em et~al.}(2014)Xiang, He, \& Wang]{Xiang-He-Wang-2014-10-14}
{\sc Xiang, S., He, G. \& Wang, H.} (2014)
\newblock On fast and stable implementation of {C}lenshaw--{C}urtis and
  {F}ej\'er-type quadrature rules.
\newblock {\em Abst. Appl. Anal.}, {\bf 2014}, 10.

\end{thebibliography}

\appendix

\section{Complexity of $\PN{\alpha,\beta}^{\rm REC}$}\label{appendix:complexity}

In this section, we derive refined estimates on the complexity of applying the matrix $\PN{\alpha,\beta}^{\rm REC}$. By artificially partitioning the matrix into rectangular regions, we need to estimate~\citep[\S 3.3]{Hale-Townsend-36-A148-14}:
\begin{equation}
\sum_{k=1}^{K-1}\alpha_N^k N (i_{k+1}^1-i_k^1 + i_k^2-i_{k+1}^2),
\end{equation}
to leading order. Fortunately, as $\theta\to0$ or $\theta\to\pi$, $g_M(\theta)$ is its own asymptotic expansion. For brevity, we derive the leading order asymptotics of $i_k^1$, and deduce those of $i_k^2$ by symmetry. To leading order:
\begin{equation}
g_M(\theta) \sim g_M^1(\theta) = \dfrac{(\frac{1}{2}+\alpha)_M(\frac{1}{2}-\alpha)_M}{M!}\dfrac{1}{\sin^{M+\alpha+\frac{1}{2}}\frac{\theta}{2}},\quad{\rm as}\quad\theta\to0.
\end{equation}
Then, to determine the leading order estimate of $i_k^1$:
\begin{equation}
\varepsilon \sim 2C_{j_k,M}^{\alpha,\beta}g_M^1\left(\dfrac{i_k^1\pi}{N+1}\right),
\end{equation}
or isolating for $i_k^1$:
\begin{equation}
i_k^1 \sim \left\lfloor \dfrac{2(N+1)}{\pi}\sin^{-1}\left(\left(\dfrac{(\frac{1}{2}+\alpha)_M(\frac{1}{2}-\alpha)_M}{\varepsilon M! 2^{2M-1}\sqrt{\pi}j_k^{M+\frac{1}{2}}}\right)^{\frac{1}{M+\alpha+\frac{1}{2}}}\right) \right\rfloor,\quad{\rm as}\quad N\to\infty.
\end{equation}
Using the fact that $\sin^{-1}x \sim x$ as $x\to0$, we find:
\begin{equation}
i_k^1 = {\cal O}\left(N\times j_k^{-\frac{M+\frac{1}{2}}{M+\alpha+\frac{1}{2}}}\right) = {\cal O}\left(N^{\frac{\alpha}{M+\alpha+\frac{1}{2}}}\times \alpha_N^{-k\frac{M+\frac{1}{2}}{M+\alpha+\frac{1}{2}}}\right),\quad{\rm as}\quad N\to\infty.
\end{equation}
Therefore, the sum involving $i_{k+1}^1$ and $i_k^1$ is, to leading order:
\begin{align}
\sum_{k=1}^{K-1} \alpha_N^k N (i_{k+1}^1-i_k^1) & = {\cal O}\left(N^{\frac{M+2\alpha+\frac{1}{2}}{M+\alpha+\frac{1}{2}}}\sum_{k=1}^{K-1} \left(\alpha_N^{-\frac{M+\frac{1}{2}}{M+\alpha+\frac{1}{2}}}-1\right)\alpha_N^{k\frac{\alpha}{M+\alpha+\frac{1}{2}}}\right),\\
& = {\cal O}\left(\dfrac{K N^{\frac{M+2\alpha+\frac{1}{2}}{M+\alpha+\frac{1}{2}}}}{\alpha_N^{\frac{M-\alpha+\frac{1}{2}}{M+\alpha+\frac{1}{2}}}}\right),\quad{\rm as}\quad N\to\infty.
\end{align}
By the symmetry in $\alpha\leftrightarrow\beta$ and $\theta\leftrightarrow\pi-\theta$, we have:
\begin{equation}
\sum_{k=1}^{K-1} \alpha_N^k N (i_k^2-i_{k+1}^2) = {\cal O}\left(\dfrac{K N^{\frac{M+2\beta+\frac{1}{2}}{M+\beta+\frac{1}{2}}}}{\alpha_N^{\frac{M-\beta+\frac{1}{2}}{M+\beta+\frac{1}{2}}}}\right),\quad{\rm as}\quad N\to\infty.
\end{equation}
Therefore, the simplified estimate ${\cal O}(N\log^2 N/\log\log N)$ is a local expansion near $(\alpha,\beta)\approx(0,0)$, and we observe in Figure~\ref{fig:timings} that it holds over $(\alpha,\beta)\in(-\tfrac{1}{2},\tfrac{1}{2}]^2$ in practice, so long as $M\ge5$.

\section{Jacobi parameters resulting in reduced complexity}\label{appendix:EdgeCases}

\subsection{$\alpha=\beta=\lambda-\frac{1}{2}$}

In the case that $\alpha=\beta=\lambda-\tfrac{1}{2}$, we are a normalization away from the ultraspherical or Gegenbauer polynomials. These asymptotics are given by~\citep[\S 18.15]{Olver-et-al-NIST-10}:
\begin{equation}
P_n^{(\lambda-\frac{1}{2},\lambda-\frac{1}{2})}(\cos\theta) = \sum_{m=0}^{M-1} C_{n,m}^\lambda \dfrac{\cos\theta_{n,m}^\lambda}{\sin^{m+\lambda}\theta} + R_{n,M}^\lambda(\theta).
\end{equation}
Here, we have:
\begin{align}
C_{n,m}^\lambda & = \dfrac{2^\lambda \Gamma(n+\lambda+\tfrac{1}{2})}{\sqrt{\pi} \Gamma(n+\lambda+1)} \dfrac{(\lambda)_m(1-\lambda)_m}{2^m m! (n+\lambda+1)_m},\\
\theta_{n,m}^\lambda & = (n+m+\lambda)\theta - (m+\lambda)\tfrac{\pi}{2} = n\theta - (m+\lambda)(\tfrac{\pi}{2}-\theta).
\end{align}
and $x=\cos\theta$. The coefficients $C_{n,m}^\lambda$ can be computed by the recurrence:
\begin{equation}
C_{n,m}^\lambda = \dfrac{(\lambda+m-1)(m-\lambda)}{2m(n+\lambda+m)}C_{n,m-1}^\lambda,\qquad C_{n,0}^\lambda = \dfrac{2^\lambda \Gamma(n+\lambda+\tfrac{1}{2})}{\sqrt{\pi} \Gamma(n+\lambda+1)}.
\end{equation}
So long as $\lambda\in(0,1)$, the error is bounded by twice the magnitude of the first neglected term in the summation:
\begin{equation}
|R_{n,M}^\lambda(\theta)| < \dfrac{2C_{n,M}^\lambda}{\sin^{M+\lambda}\theta}.
\end{equation}
Therefore, if we set the error to $\varepsilon$, this will define a curve in the $n$-$\theta$ plane for every $M$ and $\lambda$ given by:
\begin{equation}
n \approx \dfrac{n_M^\lambda}{\sin^{\frac{M+\lambda}{M+\frac{1}{2}}}\theta},\qquad n_M^\lambda = \left\lfloor\left( \dfrac{\varepsilon \sqrt{\pi} 2^M M!}{2^{\lambda+1}(\lambda)_M(1-\lambda)_M}\right)^{-\frac{1}{M+\frac{1}{2}}}\right\rfloor.
\end{equation}

\subsection{$\alpha=\frac{1}{2}$}

If $\alpha = \tfrac{1}{2}$, then the summations in the functions $f_m(\theta)$ collapse:
\begin{equation}
f_m(\theta) = \dfrac{(\tfrac{1}{2}+\beta)_{m}(\tfrac{1}{2}-\beta)_{m}}{m!}\dfrac{\cos\theta_{n,m,0}}{\sin\left(\frac{\theta}{2}\right)\cos^{m+\beta+\frac{1}{2}}\left(\frac{\theta}{2}\right)}.
\end{equation}

\subsection{$\beta=\frac{1}{2}$}

If $\beta = \tfrac{1}{2}$, then the summations in the functions $f_m(\theta)$ collapse:
\begin{equation}
f_m(\theta) = \dfrac{(\tfrac{1}{2}+\alpha)_m(\tfrac{1}{2}-\alpha)_m}{m!}\dfrac{\cos\theta_{n,m,m}}{\sin^{m+\alpha+\frac{1}{2}}\left(\frac{\theta}{2}\right)\cos\left(\frac{\theta}{2}\right)}.
\end{equation}

\end{document}